\documentclass{article}
\usepackage[top=2.54cm,bottom=2.0cm,left=2.0cm,right=2.54cm, includeheadfoot]{geometry}

\usepackage[T1]{fontenc}
\usepackage[utf8]{inputenc}
\usepackage[english]{babel}
\usepackage{enumerate}
\usepackage{multirow,booktabs}
\usepackage[table]{xcolor}
\usepackage{fullpage}
\usepackage{lastpage}
\usepackage{indentfirst}
\usepackage{bbm}

\usepackage{footnote}

\usepackage{verbatim}

\usepackage{amsmath,amsfonts,amssymb,amscd,amsthm}

\usepackage{bm}

\usepackage[all,2cell]{xy} \UseAllTwocells \SilentMatrices

\usepackage{hyperref}

\usepackage{subfiles}

\usepackage{multicol}

\newtheorem{teo}{Theorem}[section]
\newtheorem{lem}[teo]{Lemma} % use the same couter of theorem
\newtheorem{cor}[teo]{Corollary}
\newtheorem{prop}[teo]{Proposition}

\newtheorem{defn}[teo]{Definition}

\newtheorem{fat}[teo]{Fact}
\newtheorem*{claim*}{Claim}
\newtheorem{op}[teo]{Question}
\newtheorem{rem}[teo]{Remark}

\begin{document}

\title{The Galois group of a Special Group}

\author{Kaique Matias de Andrade Roberto\thanks{Institute of Mathematics and Statistics, University of São Paulo, Brazil.  Emails:   kaique.roberto@usp.br, hugomar@ime.usp.br} \ Hugo L. Mariano\thanks{The authors want to express their gratitude to Coordena\c c\~ao de Aperfei\c coamento de Pessoal de N\' ivel Superior (Capes -Brazil) by the financial support to develop this work: project Capes Math-Amsud number 88881.694471/2022-01.}
% \ (Maximo Dickmann)
}

\date{}
\maketitle

\begin{abstract}
In this ongoing work, we extend to a class of well-behaved pre-special hyperfields  the work of  J. Min\'a\v c and Spira (\cite{minac1996witt}) that describes a (pro-2)-group of a field extension that encodes the quadratic form theory of a given  field $F$: in \cite{adem1999cohomology} it is shown that its associated cohomology ring contains a copy of the cohomology ring of the field $F$. Our construction, a contravariant functor into the category of "pointed" pro-2-groups, is essentially given by  generators and relations of profinite-2-groups. We prove that such profinite groups $\mbox{Gal}(F)$ encode the space of orders of the special group canonically  associated to the hyperfield $F$  and provide a criterion to detect when $F$ is formally real or not. %Moreover, we calculate the Galois groups associated to some of the main constructions of special groups like quotients, finite products, directed colimits and extensions.
\end{abstract}

\section{Introduction}

% \linenumbers                % mostra as linhas no pdf

The Igr's functors $W_*, k_*$ were extended by M. Dickmann and F.Miraglia from the category of fields of characteristic $\neq 2$ to the category of special groups (equivalently, the category of special hyperfields). Another relevant Igr functor, the graded cohomology ring, $H^*(Gal(F^s|F), \{\pm1\})$ remains defined only on the field setting. This chapter constitutes an attempt to provide an Igr functor associated to a (Galois) cohomology theory for  special groups, based on the work of J. Minac and M. Spira \cite{minac1996witt}. We will  define - by "generator and relations", $\mbox{Gal}(G)$, the {\em Galois Group of a special group $G$}, and provide some  properties of this construction, as the encoding of the orderings on $G$. However, since deeper results will depend of a description of $\mbox{Gal}(G)$ "from below", and it still unavailable a complete theory of algebraic extension of (super)hyperfields, we will not pursue a more complete development of this cohomology theory in this work, reserving it for a future research.

 We will work in the category of pro-2-groups, and take, as usual, the conventions: "subgroup" means "closed subgroup"; "subgroup generated by a subset" means "the closure of the abstract group generated by the subset"; "morphism" means "continuous homomorphism".

 The relationship between Galois groups of fields with orderings and
quadratic forms, established by the works of Artin-Schreier (1920's) and Witt (late
1930's) are reinforced by a seminal paper of John Milnor (\cite{milnor1970algebraick}, 1971) through the definition
of a (mod 2) k-theory graded ring that "interpolates" the graded Witt ring and the
cohomology ring of fields: the three graded rings constructions determine functors
from the category of fields where 2 is invertible that, almost tree decades later, are
proved to be naturally isomorphic by the work of Voevodsky with co-authors.

Since the 1980's, have appeared many abstract approaches to the algebraic theory
of quadratic forms over fields that are essentially equivalent (or dually equivalent):
between them we emphasize the theory of Special Groups developed by
Dickmann-Miraglia: a first-order axiomatization of the algebraic theory of quadratic forms (see \cite{dickmann2000special}, \cite{dickmann1998quadratic}, \cite{dickmann2003lam}), in this setting the quadratic form behavior of a non 2 characteristic field $F$ is faithfully encoded by an structure $G(F)$ consisting of an (multiplicative) exponent  two group $F^{\times}/F^{\times 2}$ with a distinguished element $-1/F^{\times 2}$ and a quaternary relation that describes the isometry relation between two diagonal quadratic forms of dimension 2. The notions of (graded) Witt rings and k-theory are extended to the category of Special groups with remarkable pay-offs on questions on quadratic forms over fields.   In \cite{dickmann1998quadratic} they get a proof of Marshall's signature conjecture for Pythagorean by the relation between three functors associated with the quadratic form theory of a field $F$ as described by Milnor (\cite{milnor1970algebraick}): the mod 2 k-theory functor of the field ($k_{\ast}(F)$) is a source of two canonical arrows one for the graduated  Witt Ring of the field ($(s_{n}(F) : k_{n}(F) \rightarrow (I^{n}(F)/I^{n+1}(F)))_{n \in \mathbb N}$) and the other one to the cohomology ring of the absolute Galois group of $F$ with coefficients $\mathbb Z_d$ ($h_{n}(F) : k_{n}(F) \rightarrow H^{n}(gal(F^{s}|F),\mathbb Z_d))_{n \in \mathbb N}$). In \cite{dickmann2000special} and \cite{dickmann2003lam} they presented constructions that generalizes, to the wider context of special groups, the first two constructions and obtained a similar natural connection between them:  ($(s_{n}(G) : k_{n}(G) \rightarrow (I^{n}(F)/I^{n+1}(G)))_{n \in \mathbb N}$), whenever $G$ is a special group.

 In the work of Minac-Spira \cite{minac1996witt}, for each field $F$ with $char(F) \neq 2$ they prescribe  an Galois extension
$F^{(3)}|F$ that is a low subfield in the genetic way to build the quadratic closure $F^{q}$ of $F$ \footnote{Let $F^{(2)} = F(\sqrt{a}, a \in F)$ = compositum in $F^q$ of quadratic extensions of $F$ then $F^{(3)} = F^{(2)}(\sqrt{y},$ such that $F^{(2)}(\sqrt{y}|F)$ is Galois)} encodes much of the quadratic forms information about $F$.
They showed that the Galois group of that extensions can be described by "generators and relations" given by the dimension of the $\mathbb Z_d$-vector space
${F^\times}/{F^\times}^2$ :
that is the starting point of the present paper. In \cite{adem1999cohomology}
%( [{\bf AKM}] {AKM} A. Adem, D. B. Karagueuzian, J. Min\'a\v{c}, {\em On the Cohomology of Galois Groups Determined by Witt Rings}, Advances in Mathematics {\bf 148} (1999), 105-160.)
it is proved that the cohomology ring of the Galois group of the extension $F^{(3)}|F$ with coefficients $\mathbb Z_d$ contains the cohomology ring of the absolute Galois group of $F$ with coefficients $\mathbb Z_d$  and is the target of a canonical morphism of graduated rings from the Milnor's mod 2 reduced $K$-theory ring ($h_{n}(F) : k_{n}(F) \rightarrow H^{n}(gal(F^{s}|F),\mathbb Z_d))_{n \in \mathbb N}$).

A (commutative, unital) multiring is essentialy a (commutative, unital) ring with multivalued sum  and satisfying a weak distributivity law (\cite{marshall2006real}); multirings that satisfy a full  distributivity law are known as hyperrings. In \cite{ribeiro2016functorial} and \cite{roberto2021quadratic} are described encodings (i.e., equivalences of categories) between the category of special groups (resp. pre-special groups; reduced special groups) and the category of special hyperfields (resp. pre-special hyperfields, real reduced hyperfields).
Elements of a extension theory of hyperfields are developed in \cite{roberto2021ACmultifields1}, \cite{roberto2022algebraic} and \cite{roberto2021hauptsatz}.

In \cite{roberto2021ktheory} is developed a k-theory for a category of hyperfields that contains the category of pre-special hyperfields, this generalizes simultaneously Milnor's k-theory for fields (\cite{milnor1970algebraick}) and Dickmann-Miraglia's k-theory for special groups (\cite{dickmann2006algebraic}). In \cite{roberto2022graded} is provided a detailed analysis of some categories of inductive graded ring (IGR) - a concept introduced in \cite{dickmann1998quadratic} in order to provide a solution of Marshall's signature conjecture for Pythagorean fields; moreover the functor of k-theory of a pre-special hyperfields is a (free) IGR.

In the present work we extend  to a class of (well-behaved) pre-special hyperfields  the work of  J. Min\'a\v c and Spira (\cite{minac1996witt}) that describes a (pro-2)-group of a field extension that encodes the quadratic form theory of a given  field $F$: in \cite{adem1999cohomology} it is shown that its associated cohomology ring is contains  a copy of the cohomology ring of the field $F$.
Our construction, a contravariant functor into the category of "pointed" pro-2-groups, is essentially given by  generators and relations of profinite-2-groups. We prove that such profinite groups $gal(F)$ encode the space of orders of the special group canonically  associated to the hyperfield $F$  and provide a criteria to detect when $F$ is formally real or not. Moreover, we calculate the Galois groups associated to some of the main constructions of special groups  like quotients, finite products, directed colimits and extensions.

 A next step is to develop and  understand the associated cohomology ring, relating it with k-theory functor of the special groups, and apply such machinery to questions on Special Groups theory. Our final goal is to develop cohomological methods in SG-Theory, and apply such machinery to questions on Special Groups theory. In particular, can we describe "cohomological obstructions" for: (i) a reduced SG to satisfy Marshall's signature conjecture; (ii) a formally real SG to satisfy Lam's conjecture?  It should be emphasized that is a SG does not satisfies any of the above properties, then it cannot be isomorphic to a special group associated to a field, a major question in the SG-theory.

% {\bf Outline of  the work}: ...

\section{The motivation: W-groups}

The context that we will keep in mind is essentially that of the results developed in Sections 1 and 2 of \cite{minac1996witt}. In this Section we will reproduce (and expand the details of) part of these results.

Consider a field $F$. In \cite{minac1996witt}, J. Minac and M. Spira define a special Galois extension of the base field $F$, and determine its structure and its Galois group through the behavior of quaternions algebras over $F$. As they developed in \cite{minac1996witt}, this extension contain essentially all the information need to understand the behavior of quadratic forms over $F$.

Recall that the \textbf{quadratic closure} of $F$, denoted by $F_q$, is the smallest extension of $F$ which is closed under taking of square roots (or more explicit, the compositum of all 2-towers over $F$ inside a fixed algebraic closure of $F$). The group $\mbox{Gal}_F(F_q)$ will be
denoted by $G^q_F$.

Let $\{a_i:i\in I\}$ be a basis of $\dot F/\dot F^2$ (as Minac and Spira did in their paper, we will assume that $1,2,..,n\in I$ with $1<2<...<n$ in order to easy our presentation). We define $F^{(2)}=F(\sqrt{a}:a\in\dot F)$ (note that $F^{(2)}=F(\sqrt{a_i}:i\in I$), $\mathcal E=\{y\in F^{(2)}:F^{(2)}(\sqrt y)|F\mbox{ is Galois}\}$ and $F^{(3)}=F^{(2)}(\sqrt{y}:y\in\mathcal E)$ if $F$; if $F$ is quadratically closed we set $F^{(2)}=F^{(3)}=F$.

Minac and Spira built a strong connection between $F^{(3)}$ and the Witt ring of $F$. They named $F^{(3)}$ as \textbf{Witt closure} of $F$. The group $G_F:=\mbox{Gal}_F(F^{3}|F)$ is called the \textbf{W-group} of $F$. Our goal here is to describe a way to factor $G_F$ as $G_F\cong\mathcal W(I)/\mathcal V(I)$, with $\mathcal W(I)$ and $\mathcal V(I)$ interesting profinite groups. This procedure will reveal how to generalize $G_F$ in the context of abstract theories of quadratic forms, and in particular, describe what would be a Galois group associated to a special group. The first step is to describe $\mathcal W(I)$.

For an arbitrary group $G$, define $\hat G=G^4[G^2,G]$, i.e, the (closed) subgroup generated by fourth powers and by commutators of the form $[g^2,h]$ for $g,h\in G$. Let $t^4[g^2,h]\in\hat G$. Then, for each  $z\in G$:
\begin{align*}
 z^{-1}(t^4[g^2,h])z&=z^{-1}t^4[g^2,h]z=(z^{-1}t^4z)(z^{-1}[g^2,h]z)
\end{align*}
with $z^{-1}t^4z=(z^{-1}tz)^4\in G^4$ and $z^{-1}[g^2,h]z\in [G^2,G]$, because
\begin{align*}
 z^{-1}[g^2,h]z&=z^{-1}g^{-2}h^{-1}g^2hz \\
 &=(z^{-1}g^{-2}z)(z^{-1}h^{-1}z)(z^{-1}g^2z)(z^{-1}hz) \\
 &=(z^{-1}gz)^{-2}(z^{-1}hz)^{-1}(z^{-1}gz)^2(z^{-1}hz) \\
 &=[(z^{-1}gz)^2,(z^{-1}hz)].
\end{align*}
Hence $\hat G$ is a normal subgroup of $G$, and we define $\overline G=G/\hat G$. Let $\mathcal C$ denote the class of profinite 2-groups $G$ such that $\hat G=\{1\}$.

 The main example (and the motivation to consider this full subcategory of pro-2-groups) is the following fact:  If $char(F) \neq 2$ then $G=Gal(F^{(3)}|F)$ satisfies this condition $\hat G=\{1\}$, since, by Proposition 2.1 in \cite{minac1996witt}, $G \cong G_q/{G_q}^4.[{G_q}^2,{G_q}]$, where $G_q=Gal(F_q|F)$ and $F_q$ is a quadratic closure of $F$ .

A pro-$\mathcal C$-group will be called just $\mathcal C$-group, and $\mathcal C$-group on $I$ if it has a minimal set of generators of cardinality $|I|$.

Let $I$ be a well-ordered set. The next step is to describe a canonical way to represent the elements of $\overline S$ where $S$ is the free pro-2-group on a nonempty set $I$. Let
$$\mathcal W(I):=\prod_{i\in I}\mathbb Z_2\times\prod_{\substack{i,j\in I \\ i<j}}\mathbb Z_2\times\prod_{i\in I}\mathbb
Z/2\mathbb Z.$$
Here we are considering $\mathbb Z_2$ {\em multiplicatively}, i.e, $\mathbb Z_2\cong\{1,-1\}$. A typical element of $\mathcal W_I$ will be written as $(t_i^{\alpha_i})(t_{ij}^{\beta_{ij}})(x_i^{\gamma_i})$, where $\alpha_i,\beta_{ij},\gamma_i\in\{0,1\}$ and $t_i=t_{ij}=x_i=-1$ for all $i,j\in I$.

Let $g,h\in\mathcal W(I)$
\begin{align*}
 g&=(t_i^{\alpha_i})(t_{ij}^{\beta_{ij}})(x_i^{\gamma_i}) \\
 h&=(t_i^{\alpha'_i})(t_{ij}^{\beta'_{ij}})(x_i^{\gamma'_i}).
\end{align*}
We define
$$gh=(t_i^{\alpha_i+\alpha'_i+\gamma_i\gamma'_i})(t_{ij}^{\beta_{ij}+\beta'_{ij}+\gamma'_i\gamma_j})(x_i^{\gamma_i+\gamma'_i}),$$
where the exponents are taken modulo 2.

\begin{lem}
 With the notation described above, $\mathcal W(I)$ is a group.
\end{lem}
\begin{proof}
 Let
 \begin{align*}
 g&=(t_i^{\alpha_i})(t_{ij}^{\beta_{ij}})(x_i^{\gamma_i}) \\
 h&=(t_i^{\alpha'_i})(t_{ij}^{\beta'_{ij}})(x_i^{\gamma'_i}) \\
 k&=(t_i^{\alpha''_i})(t_{ij}^{\beta''_{ij}})(x_i^{\gamma''_i}) \\
 1&=(t_i^0)(t_{ij}^0)(x_i^0).
\end{align*}
First of all,
\begin{align*}
 g\cdot1&=(t_i^{\alpha_i+0+\gamma_i\cdot0})(t_{ij}^{\beta_{ij}+0+0\cdot\gamma_j})(x_i^{\gamma_i+0})=g,\\
 1\cdot g&=(t_i^{0+\alpha_i+0\cdot\gamma_i})(t_{ij}^{0+\beta_{ij}+\gamma_i\cdot0})(x_i^{0+\gamma_i})=g,
\end{align*}
hence $1$ is in fact the neutral element of $\mathcal W(I)$. In order to find $g^{-1}$, we need to solve a system of modulo 2 congruences.
\begin{align*}
 gh:=1\Rightarrow
\begin{cases}\alpha_i+\alpha'_i+\gamma_i\gamma'_i\equiv_20 \\ \beta_{ij}+\beta'_{ij}+\gamma'_i\gamma_j\equiv_20 \\
\gamma_i+\gamma'_i\equiv_20 \end{cases}\Rightarrow
\begin{cases}\gamma'_i\equiv_2\gamma_i \\ \beta'_{ij}\equiv_2\beta_{ij}+\gamma_i\gamma_j \\ \alpha'_i\equiv_2\alpha_i+\gamma_i \end{cases}
\end{align*}
Then, taking
$$g^{-1}=(t_i^{\alpha_i+\gamma_i})(t_{ij}^{\beta_{ij}+\gamma_i\gamma_j})(x_i^{\gamma_i})$$
we obtain $g\cdot g^{-1}=g^{-1}\cdot g=1$.

Finally, for associativity, we have
\begin{align*}
 (g\cdot h)\cdot k&=[(t_i^{\alpha_i+\alpha'_i+\gamma_i\gamma'_i})(t_{ij}^{\beta_{ij}+\beta'_{ij}+\gamma'_i\gamma_j})(x_i^{\gamma_i+\gamma'_i})]
\cdot k \\
&=(t_i^{\alpha_i+\alpha'_i+\alpha''_i+\gamma_i\gamma'_i+(\gamma_i+\gamma'_i)\gamma''_i})
(t_{ij}^{\beta_{ij}+\beta'_{ij}+\beta''_{ij}+\gamma'_i\gamma_j+\gamma''_i(\gamma_j+\gamma'_j)})
(x_i^{\gamma_i+\gamma'_i+\gamma''_i}) \\
&=(t_i^{\alpha_i+\alpha'_i+\alpha''_i+\gamma_i\gamma'_i+\gamma_i\gamma''_i+\gamma'_i\gamma''_i})
(t_{ij}^{\beta_{ij}+\beta'_{ij}+\beta''_{ij}+\gamma'_i\gamma_j+\gamma''_i\gamma_j+\gamma''_i\gamma'_j})
(x_i^{\gamma_i+\gamma'_i+\gamma''_i})
\end{align*}
and
\begin{align*}
 g\cdot (h\cdot k)&=g\cdot
[(t_i^{\alpha_i'+\alpha''_i+\gamma'_i\gamma''_i})(t_{ij}^{\beta'_{ij}+\beta''_{ij}+\gamma''_i\gamma'_j})(x_i^{\gamma'_i+\gamma''_i})] \\
&=(t_i^{\alpha_i+\alpha'_i+\alpha''_i+\gamma'_i\gamma''_i+\gamma_i(\gamma'_i+\gamma''_i)})
(t_{ij}^{\beta_{ij}+\beta'_{ij}+\beta''_{ij}+\gamma''_i\gamma'_j+(\gamma'_i+\gamma''_i)\gamma_j})
(x_i^{\gamma_i+\gamma'_i+\gamma''_i}) \\
&=(t_i^{\alpha_i+\alpha'_i+\alpha''_i+\gamma'_i\gamma''_i+\gamma_i\gamma'_i+\gamma_i\gamma''_i})
(t_{ij}^{\beta_{ij}+\beta'_{ij}+\beta''_{ij}+\gamma''_i\gamma'_j+\gamma'_i\gamma_j+\gamma''_i\gamma_j})
(x_i^{\gamma_i+\gamma'_i+\gamma''_i}).
\end{align*}
Thus $(g\cdot h)\cdot k=g\cdot (h\cdot k)$ completing the proof.
\end{proof}

\begin{rem}
$ $
\begin{enumerate}[a -]
    \item Note that, if $|I|=n$, then
$$|\mathcal W(I)|=\left|\prod_{i\in I}\mathbb Z_2\times\prod_{\substack{i,j\in I \\ i\le j}}\mathbb Z_2\times\prod_{i\in
I}\mathbb Z_2 \right|=2^{(n^2+3n)/2}.$$

\item The example above is the free $\mathcal C$ in $n$-generators. In fact: for each $n \in \mathbb{N}$, let $F(n)$ be the free group in $n$-generators, then  $\mathcal W(n)\cong F(n)/\hat{F}(n)$.

\item It follows that the category of finite (discrete) groups $G$ with $\hat{G} = \{1\}$ is a category of $\mathbb Z_2$-modules that is closed by homomorphic images, subgroups and finite products. In particular, $\mathbb{Z}_2 = \mathbb Z/2\mathbb Z, \mathbb{Z}_4 = \mathbb Z/4\mathbb Z$ and $\mathbb{D}_4 \cong \mathbb{Z}_2 \ltimes \mathbb{Z}_4$ (the 8-element dihedral group), are finite $\mathcal C$-groups.
\end{enumerate}
\end{rem}

Lets denote, for $k,l\in I$,
\begin{align*}
 t_k&:=(t_i^{\delta_{ik}})(1_{ij})(1_i) \\
 t_{kl}&:=(1_i)(t_{ij}^{\delta_{(ij)(kl)}})(1_i) \\
 x_k&:=(1_i)(1_{ij})(x_i^{\delta_{ik}})
\end{align*}
where for all $i,j,k,l\in I$, $\delta_{ik}=1$ if $i=k$ and $\delta_{ik}=0$ otherwise; and $\delta_{(ij)(kl)}=1$ if $i=k$ and $j=l$, and $\delta_{(ij)(kl)}=0$ otherwise. After some straightforward calculations we obtain the following results.

\begin{lem}\label{fixms1}
 Consider $t_k,t_{kl},x_k$ as above. Then for all $g,h\in\mathcal W(I)$, with $g=(t_i^{\alpha_i})(t_{ij}^{\beta_{ij}})(x_i^{\gamma_i})$, $h=(t_i^{\alpha'_i})(t_{ij}^{\beta'_{ij}})(x_i^{\gamma'_i})$ and $ z=(t_i^{\alpha''_i})(t_{ij}^{\beta''_{ij}})(x_i^{\gamma''_i})$, we have the following:
\begin{enumerate}[i -]
 \item $t_k\cdot t_k=1$.
 \item $x_k\cdot x_k=t_k$.
 \item If $k<l$ then $[x_k,x_l]=t_{kl}$.
 \item $g^{-1}=(t_i^{\alpha_i+\gamma_i})(t_{ij}^{\beta_{ij}+\gamma_i\gamma_j})(x_i^{\gamma_i})$.
 \item $g^2=(t_i^{\gamma_i})(t_{ij}^{\gamma_i\gamma_j})(1_i)$.
 \item $h^g=ghg^{-1}=(t_i^{\alpha'_i})(t_{ij}^{\beta'_{ij}+\gamma_i\gamma'_j+\gamma'_i\gamma_j})(x_i^{\gamma'_i})$.
 \item $[g,h]=(1_i)(t_{ij}^{\gamma_i\gamma'_j+\gamma'_i\gamma_j})(1_j)$.
 \item $g^4=[g^2,h]=1$.
 \item $[[z,w],h]=1$.
\end{enumerate}
\end{lem}
%\begin{proof}
 %$ $
 %\begin{enumerate}[i -]
  %   \item By direct calculation we get
   %  \begin{align*}
    %     t_k\cdot t_k=[(t_i^{\delta_{ik}})(1_{ij})(1_i)]\cdot[(t_i^{\delta_{ik}})(1_{ij})(1_i)]=(t_i^{\delta_{ik}+\delta_{ik}+0\cdot0})(1_{ij})(1_i)=(t_i^0)(1_{ij})(1_i)=1.
     %\end{align*}
     %\item By direct calculation we get
     %\begin{align*}
      %   x_k\cdot x_k=[(1_i)(1_{ij})(x_i^{\delta_{ik}})]\cdot
       %  [(1_i)(1_{ij})(x_i^{\delta_{ik}})]=
        % (t_i^{\delta_{ik}\delta_{ik}})(t_{ij}^{\delta_{ik}\delta_{jk}})(x_i^{\delta_{ik}+\delta_{ik}})=(t_i^{\delta_{ik}})(1_{ij})(1_i)=t_k.
     %\end{align*}
     %\item By direct calculation we get
     %\begin{align*}
        % x_k\cdot x_l=[(1_i)(1_{ij})(x_i^{\delta_{ik}})]\cdot
        % [(1_i)(1_{ij})(x_i^{\delta_{il}})]=
        % (t_i^{\delta_{ik}\delta_{il}})(t_{ij}^{\delta_{ik}\delta_{jl}})(x_i%^{\delta_{ik}+\delta_{il}}).
     %\end{align*}
     %Since $\delta_{ik}\delta_{jl}=1$ if and only if $i=k$ and $j=l$, we get
     %$$x_k\cdot x_l=(t_i^{\delta_{ik}\delta_{il}})(t_{ij}^{\delta_{ik}\delta_%{jl}})(x_i^{\delta_{ik}+\delta_{il}})
%     =(1_i)(t_{ij}^{\delta_{ik}\delta_{jl}})(1_i)=t_{kl}.$$
%     \item We already did in Lemma \ref{fixms1}.
%     \item
% \end{enumerate}
%\end{proof}

\begin{rem}
In \cite{minac1996witt}, they simply denote
$$g=(x_i^{2\alpha_i})([x_i,x_j]^{\beta_{ij}})(x_i^{\gamma_i}).$$
But we are not going to use this simplification here.
%However, in order to easy comparison with results here and there, we summarize the formulas in terms of both notations: for $g,h\in\mathcal W(I)$, with $g=(t_i^{\alpha_i})(t_{ij}^{\beta_{ij}})(x_i^{\gamma_i})$,
 %$h=(t_i^{\alpha'_i})(t_{ij}^{\beta'_{ij}})(x_i^{\gamma'_i})$ and $ z=(t_i^{\alpha''_i})(t_{ij}^{\beta''_{ij}})(x_i^{\gamma''_i})$, we have:
 %\begin{align*}
  %   gh&=(t_i^{\alpha_i+\alpha'_i+\gamma_i\gamma'_i})(t_{ij}^{\beta_{ij}+\beta'_{ij}+\gamma'_i\gamma_j})(x_i^{\gamma_i+\gamma'_i})=
   %  (x_i^{2(\alpha_i+\alpha'_i+\gamma_i\gamma'_i)})([x_i,x_j]^{\beta_{ij}+\beta'_{ij}+\gamma'_i\gamma_j})(x_i^{\gamma_i+\gamma'_i})\\
    % g^{-1}&=(t_i^{\alpha_i+\gamma_i})(t_{ij}^{\beta_{ij}+\gamma_i\gamma_j})(x_i^{\gamma_i})=(x_i^{2(\alpha_i+\gamma_i)})([x_i,x_j]^{\beta_{ij}+\gamma_i\gamma_j})(x_i^{\gamma_i})\\
     %g^2&==\\
     %h^g&=ghg^{-1}==\\
     %[g,h]&=
 %\end{align*}
%where the exponents are taken modulo 2.
\end{rem}

Now, for each $i,j\in I$, consider the following three sets:
\begin{align*}
 M_i&:=\left\lbrace g=(t_i^{\alpha_i})(t_{ij}^{\beta_{ij}})(x_i^{\gamma_i})
\in\mathcal W(I):\gamma_i=0\right\rbrace, \\
 S_i&:=\left\lbrace g=(t_i^{\alpha_i})(t_{ij}^{\beta_{ij}})(x_i^{\gamma_i})
\in\mathcal W(I):\alpha_i=\gamma_i=0\right\rbrace, \\
 D_{ij}&:=\left\lbrace g=(t_i^{\alpha_i})(t_{ij}^{\beta_{ij}})(x_i^{\gamma_i})
\in\mathcal W(I):\beta_{ij}=\gamma_i=\gamma_j=0\right\rbrace.
\end{align*}
Now, consider the following families
\begin{align*}
 M(I)&:=\{M_i: i\in I\},\, S(I):=\{S_i: i\in I\},\, D(I):=\{D_{ij} : i,j\in I,  i\le j\},\\
 V&:=M(I)\cup S(I)\cup D(I).
\end{align*}

\begin{prop}\label{fixms2}
 Let $i,j\in I$.
 \begin{enumerate}[a -]
  \item $M_i$ is a maximal normal subgroup of $\mathcal W(I)$ such that $\mathcal W(I)/M_i\cong\mathbb Z_2$.
  \item $S_i$ is a normal subgroup of $\mathcal W(I)$ such that $S_i\subseteq M_i$ and $\mathcal W(I)/S_i\cong\mathbb Z_4$.
  \item $D_{ij}$ is a normal subgroup of $\mathcal W(I)$ such that $D_{ij}\subseteq M_i\cap M_j$ and $\mathcal W(I)/D_{ij}\cong \mathbb{D}_4$.
  \item $\bigcap V=\{1\}$.
 \end{enumerate}
\end{prop}
\begin{proof}
We establish the following notation: let $g\in\mathcal W(I)$. Then $g=(t_i^{\alpha_i})(t_{ij}^{\beta_{ij}})(x_i^{\gamma_i})$ for suitable $\alpha_i,\beta_{ij},\gamma_i\in\{0,1\}$. We denote
\begin{align*} \alpha_i(g):=\alpha_i,\,\beta_{ij}(g):=\beta_{ij}\mbox{ and }\gamma_i(g):=\gamma_i.
\end{align*}
In this sense, if $g,h\in\mathcal W(I)$ then
$$gh=(t_i^{\alpha_i(g)+\alpha_i(h)+\gamma_i(g)\gamma_i(h)})(t_{ij}^{\beta_{ij}(g)+\beta_{ij}(h)+\gamma_i(h)\gamma_j(g)})(x_i^{\gamma_i(g)+\gamma_i(h)}).$$

Moreover, using the formulas in Lemma \ref{fixms1}, we obtain that for all $i,j\in I$, $M_i$, $S_i$ and $D_{ij}$ are proper normal subgroups of $\mathcal W(I)$.

\begin{enumerate}[a -]
    \item Let $\tau,\theta\in\mathcal W(I)\setminus M_i$. Then $\gamma_i(\tau)=\gamma_i(\theta)=1$ and   $$\gamma_i(\theta^{-1}\tau)=\gamma_i(\theta)+\gamma_i(\tau)=0.$$
     Therefore $\theta^{-1}\tau\in M_i$ which imply $$\mathcal W(I)/M_i=\{\overline1,\overline\tau\}\cong\mathbb Z_2.$$

    \item Note that $S_i\subseteq M_i$. Now, let $\tau,\theta\in M_i\setminus S_i$. Then $\alpha_i(\tau)=\alpha_i(\theta)=1$ and
    $$\alpha_i(\theta^{-1}\tau)=[\alpha_i(\theta)+\gamma_i(\theta)]+\alpha_i(\tau)+\gamma_i(\theta)\gamma_i(\tau)=0.$$
    Hence $\theta^{-1}\tau\in S_i$, and $M_i/S_i\cong\mathbb Z_2$. So we have an exact sequence
    $$1\rightarrow M_i/S_i\xrightarrow{\iota}\mathcal W(I)/S_i\xrightarrow{\pi}\mathcal W(I)/M_i\rightarrow1,$$
    where $\iota$ and $\pi$ are respectively the canonical inclusion and canonical projection. Moreover
    $$\mathcal W(I)/S_i\cong
    M_i/S_i\rtimes \mathcal W(I)/M_i
    \mbox{ or }
    \mathcal W(I)/S_i\cong
    \mathcal W(I)/M_i\ltimes M_i/S_i.$$
    In both cases,
    $$|\mathcal W(I)/S_i|=|M_i/S_i\times \mathcal W(I)/M_i|=|\mathbb Z_2\times\mathbb Z_2|=4.$$
    Now let $\sigma\in\mathcal M\setminus M_i$. We have $\sigma^4=1\in S_i$ with
    \begin{align*}
      \gamma_i(\sigma^3)&:=\gamma_i(\sigma)=1.
    \end{align*}
    Then $\overline\sigma^3\ne\overline1$ in $\mathcal W(I)/S_i$, which proves that $\mathcal W(I)/S_i$ has an element of order 4. Then
    $$\mathcal W(I)/S_i\cong\mathbb Z_4.$$

    \item Remember that
    $$\mathbb{D}_4:=\langle r,s:r^4=s^2=(sr)^2=1\rangle.$$
    Using the same argument of item (b), we get $(M_i\cap M_j)/D_{ij}\cong\mathbb Z_2$ and $|\mathcal W(I)/(M_i\cap M_j)|=4$ with
    $$|\mathcal W(I)/D_{ij}|=|(M_i\cap M_j)/D_{ij}\times \mathcal W(I)/(M_i\cap M_j)|=8.$$
    More specifically, if we get $\tau_1\in (M_i\cap M_j)\setminus D_{ij}$, $\tau_2,\theta_2\in M_i\setminus M_j$, $\tau_3,\theta_3\in M_j\setminus M_i$ and $\tau_4,\theta_4\in\mathcal W(I)\setminus(M_i\cup M_j)$, with
    \begin{align*}
        &\beta_{ij}(\tau_2)=\beta_{ij}(\tau_3)=\beta_{ij}(\tau_4)=1 \\
        &\beta_{ij}(\theta_2)=\theta_{ij}(\theta_3)=\theta_{ij}(\tau_4)=0
    \end{align*}
    then the following equations hold in $\mathcal W(I)/D_{ij}$
    \begin{align*}
    &\overline{\tau_1}^2=\overline1\\
        &\overline{\tau^{-1}_4\theta_4}=\overline{\tau^{-1}_3\theta_3}=\overline{\tau^{-1}_2\theta_2}=\overline{\tau_1}
    \end{align*}
    Then
    $$\mathcal W(I)/D_{ij}=\{\overline{1},\overline{\tau_1},\overline{\tau_2},\overline{\tau_3},\overline{\tau_4},\overline{\tau_1\tau_2},\overline{\tau_1\tau_3},\overline{\tau_1\tau_4}\},$$
    with the following table of multiplication:
    \begin{center}
\begin{tabular}{|l|l|l|l|l|l|l|l|l|}
\hline
$\cdot$ & $\overline{\tau_1}$ & $\overline{\tau_2}$ & $\overline{\tau_3}$ & $\overline{\tau_4}$ & $\overline{\tau_1\tau_2}$ & $\overline{\tau_1\tau_3}$ & $\overline{\tau_1\tau_4}$ \\
\hline
$\overline{\tau_1}$ & $\overline{1}$ & $\overline{\tau_1\tau_2}$ & $\overline{\tau_1\tau_3}$ & $\overline{\tau_1\tau_4}$ & $\overline{\tau_2}$ & $\overline{\tau_3}$ & $\overline{\tau_4}$ \\
\hline
$\overline{\tau_2}$ & $\overline{\tau_1\tau_2}$ &  $\overline1$ & $\overline{\tau_1\tau_4}$ & $\overline{\tau_3}$ & $\overline{\tau_1}$ & $\overline{\tau_4}$ & $\overline{\tau_1\tau_3}$ \\
\hline
$\overline{\tau_3}$ & $\overline{\tau_1\tau_3}$ & $\overline{\tau_4}$ & $\overline1$ & $\overline{\tau_1\tau_2}$ & $\overline{\tau_1\tau_4}$ & $\overline{\tau_1}$ & $\overline{\tau_2}$ \\
\hline
$\overline{\tau_4}$ & $\overline{\tau_1\tau_4}$ & $\overline{\tau_1\tau_3}$ & $\overline{\tau_2}$ & $\overline{\tau_1}$ & $\overline{\tau_3}$ & $\overline{\tau_1\tau_2}$ & $\overline1$ \\
\hline
$\overline{\tau_1\tau_2}$ & $\overline{\tau_2}$ & $\overline{\tau_1}$ & $\overline{\tau_4}$ & $\overline{\tau_1\tau_3}$ & $\overline1$ & $\overline{\tau_1\tau_4}$ & $\overline{\tau_3}$ \\
\hline
$\overline{\tau_1\tau_3}$ & $\overline{\tau_3}$ & $\overline{\tau_1\tau_4}$ & $\overline{\tau_1}$ & $\overline{\tau_2}$ & $\overline{\tau_4}$ & $\overline1$ & $\overline{\tau_1\tau_2}$ \\
\hline
$\overline{\tau_1\tau_4}$ & $\overline{\tau_4}$ & $\overline{\tau_3}$ & $\overline{\tau_1\tau_2}$ & $\overline1$ & $\overline{\tau_1\tau_3}$ & $\overline{\tau_2}$ & $\overline{\tau_1}$ \\
\hline
\end{tabular}
\end{center}
Then denoting $r=\tau_4$ and $s=\tau_2$ we get
$r^4=s^2=(sr)^2=1$ and
\begin{align*}
    \overline1&=s^2=r^4,\\
    \overline{\tau_1}&=r^2,\\
    \overline{\tau_2}&=s,\\
    \overline{\tau_3}&=sr,\\
    \overline{\tau_4}&=r,\\
    \overline{\tau_1\tau_2}&=r^2s,\\
    \overline{\tau_1\tau_3}&=r^2sr=rs,\\
    \overline{\tau_1\tau_4}&=r^3;
\end{align*}
witnessing the desired isomorphism.

    \item By the very definition
    \begin{align*}
        \bigcap V=\left\lbrace g=(t_i^{\alpha_i})(t_{ij}^{\beta_{ij}})(x_i^{\gamma_i})\in\mathcal W(I):\beta_{ij}=\alpha_i=\gamma_i=\gamma_j=0\mbox{ for all }i,j\in I\right\rbrace=\{1\}.
    \end{align*}
\end{enumerate}
\end{proof}

Let $V$ as in Proposition \ref{fixms2} and let $P_{fin}(V)$ be the set of finite subsets of $V$ and for $A\in P_{fin}(V)$, denote
$$X_A=W_I/\bigcap A.$$
Note that $P_{fin}(V)$ is a directed poset with the partial ordering induced by inclusion. If $B\subseteq A$, denote by $\pi_{AB}:X_A\rightarrow X_B$ the canonical projection. Then $(X_A,\pi_{AB},P_{fin}(V))$ is a projective system, in the sense that $\pi_{AA} = id_{X_{A}}$ and, if $E\subseteq B\subseteq A$, then $\pi_{AE}=\pi_{BE}\circ\pi_{AB}$.

\begin{prop}\label{fixms3}
 The canonical ``diagonal'' function
 $$\mathcal W(I)\rightarrow\varprojlim_{A\in P_{fin}(V)}X_A,\mbox{ which is given by the rule }g\mapsto\left(g/X_A\right)_{A\in P_{fin}(V)}$$
 is an abstract group isomorphism, so, by transport, $\mathcal W(I)$ is a (topological) profinite 2-group with $P_{fin}(V)$ as fundamental system of clopen neighborhoods of $\{1\}$.
\end{prop}
\begin{proof}
 Let $X=\prod_{A\in P_{fin}(V)}X_A$ and $\pi_A:X\rightarrow X_A$ be the canonical projection. Denote $\Delta:\mathcal W(I)\rightarrow X$ the morphism given by the rule $\Delta(g):=\left(g/X_A\right)_{A\in P_{fin}(V)}$. This morphism $\Delta$ is injective since
 $$\mbox{Ker}(\Delta)=\bigcap_{A\in P_{fin}(V)}=\{1\}.$$
 Now, let $\overline g=\left(g/X_A\right)_{A\in P_{fin}(V)}\in\mbox{Im}(\Delta)$. If $B\subseteq A$, we get
 $$\overline g_B=g/X_B=(g/X_A)/X_B=(\pi_{AB}(g))/X_B.$$
 Moreover $\mbox{Im}(\Delta)\subseteq\varprojlim_{A\in P_{fin}(V)}X_A$. To prove the surjectivity of $\Delta$, consider the morphism $\pi_A:\mathcal W(I)\rightarrow X_A$ given by the canonical projection. Then $(\mathcal W(I),\pi_A)$ is a compatible system of surjective morphisms where $\Delta$ is the exact morphism induced by $(\mathcal W(I),\pi_A)$. Then $\mbox{Im}(\Delta)$ is a dense subset of $\varprojlim_{A\in P_{fin}(V)}X_A$ (for instance, see Lemma 1.1.7 of \cite{ribes2000profinite}) which is also closed. If $\varphi_A:\varprojlim_{A\in P_{fin}(V)}X_A\rightarrow X_A$ denote the projection, we have a new projective system $(\varphi_A(\mbox{Im}(\Delta),\pi_{AB}|_{\varphi_A(\mbox{Im}(\Delta)})$. Then (using Corollary 1.1.8 of \cite{ribes2000profinite}) we get
 $$\mbox{Im}(\Delta)=\varprojlim_{A\in P_{fin}(V)}\varphi_A(\mbox{Im}(\Delta)=\overline{\mbox{Im}(\Delta)}=\varprojlim_{A\in P_{fin}(V)}X_A.$$
 Moreover
 $$\{\mbox{Ker}(\pi_A):\mathcal W(I)\rightarrow X_A\}_{A\in P_{fin}(V)}=P_{fin}(V)$$
 is a fundamental system of neighborhoods of $\{1\}$.
\end{proof}

Lets invoke some terminology from the theory of profinite groups:
\begin{defn}
Let $G$ be a profinite group.
\begin{enumerate}[i -]
    \item We say that $X$ \textbf{generates $G$ as a profinite group} if $G=\overline{\langle G\rangle}$. In that case, we call $X$ a set of \textbf{topological generators} of $G$.

    \item We say that $X\subseteq G$ \textbf{converges to 1} if every open subgroup $U$ of $G$ contains all but a finite number of the elements in $X$.

    \item Let $G$ be a profinite group. The \textbf{Frattini subgroup} of $G$, notation $\Phi(G)$, is the intersection of all its maximal open subgroups.
\end{enumerate}
\end{defn}

\begin{fat}
Let $G$ be a profinite group.
\begin{enumerate}[i -]
    \item $G$ is compact, Hausdorff and totally disconnected (has a topological basis of clopens).

    \item A subgroup $U$ is open if and only if is closed of finite index (Lemma 2.1.2 of \cite{ribes2000profinite}).

    \item A closed subgroup $H$ of a profinite group $G$ is the intersection of all open subgroups of $G$ containing $H$. If $H$ is normal, then $H$ is the intersection of all open normal subgroups of $G$ containing $H$ (Proposition 2.1.4 of \cite{ribes2000profinite}).

    \item A maximal closed subgroup is necessarily open.

    \item $\Phi(G)$ is a characteristic subgroup of $G$: for every automorphism $\psi:G\rightarrow G$ of $G$ we have $\psi[\Phi(G)]=\Phi(G)$.

    \item If $h:H\rightarrow G$ is a continuous homomorphism of pro-2-groups then $h[\Phi(H)]\subseteq\Phi[G]$.
\end{enumerate}
\end{fat}

\begin{lem}[Lemma 2.8.7 of \cite{ribes2000profinite}]\label{rz2.8.7}
 Let $p$ be a prime number and let $G$ be a pro-$p$ group.
 \begin{enumerate}[a -]
     \item Every maximal closed subgroup $M$ of $G$ has index $p$.
     \item The Frattini quotient $G/\Phi(G)$ is a $p$-elementary abelian profinite group and hence a vector space of the field $\mathbb F_p$ with $p$ elements.
     \item $\Phi(G)=\overline{G^p[G,G]}$, where $G^p=\{x^p:x\in G\}$ and $[G,G]$ denotes the commutator subgroup of $G$.
 \end{enumerate}
\end{lem}

\begin{lem} \label{lifting-le} Let ${\mathcal G}_{i}$,  $i=0,1$ be {\em projectives} profinite groups and ${\mathcal V}_{i} \subseteq {\mathcal G}_{i}$ be normal closed subgroups such that ${\mathcal V}_{i} \subseteq \Phi({\mathcal G}_{i})$. If $f : {\mathcal G}_{0}/{\mathcal V}_{0} \ \to \ {\mathcal G}_{1}/{\mathcal V}_{1}$ is an epimorphism (respectively an isomorphism) then there is some continuous homomorphism $\tilde{f} : {\mathcal G}_{0} \ \to \ {\mathcal G}_{1}$ such that $q_{1} \circ \tilde{f} = f \circ q_{0}$ where the $q_{i}$ are the projections on quotient; besides any such lifting $\tilde{f}$ is an epimorphism (resp. an isomorphism).

$$\xymatrix@!=4.5pc{
\mathcal G_0\ar[r]^{\tilde f}\ar[d]_{q_0} & \mathcal G_1\ar[d]^{q_1} \\
\mathcal G_0/\mathcal V_0\ar[r]_f & \mathcal G_1/\mathcal V_1
}$$
\end{lem}

In \cite{minac1996witt} is established  the main categorical property of $\mathcal W(I)$:

\begin{teo}[Universal Property of $\mathcal W(I)$, Theorem 1.1 of \cite{minac1996witt}]\label{fixms4}
 The group $\mathcal W(I)$ is the $\mathcal C$-free group on $I$-generators. In other words, $I\subseteq\mathcal W(I)$ is a generator set converging to 1 and if $f:I\rightarrow G$ is any function to a $\mathcal C$-group $G$ such that $f[I]\subseteq G$ converges to 1 then there is an (unique) continuous homomorphism $\tilde f:\mathcal W(I)\rightarrow G$ such that $\tilde f|_I=f$. Moreover, if $H$ is any $\mathcal C$-group then $H$ has a generator set converging to 1 of cardinality $|I|$ if and only if there is an epimorphism $\mathcal W(I)\twoheadrightarrow H$ with kernel $V$ contained in $\Phi(I)$, the Frattini subgroup of $\mathcal W(I)$.
\end{teo}

\begin{cor}
 Let $X=\{x_i:i\in I\}\subseteq\mathcal W(I)$. Then $X$ is a set of generators of $\mathcal W(I)$ converging to 1.
\end{cor}

\begin{prop}\label{fixms5}
 We have $\Phi(I)=\mathcal W^2_I$, and $\Phi(I)$ has $\{x_i,t_{ij}:i<j\in I\}$ as a minimal set of generators converging to 1. Moreover
 \begin{align*}
  \Phi(I)&=\mathcal W^2_I=\bigcap M(I) \\
  &=\left\lbrace g=(t_i^{\alpha_i})(t_{ij}^{\beta_{ij}})(x_i^{\gamma_i})\in\mathcal W(I):\gamma_i=0\mbox{ for all }i\in I\right\rbrace \\
&=\{g\in\mathcal W(I):\mbox{there exists }g_1,g_2,g_3\in\mathcal W(I)\mbox{ such that }g=g_1^2g_2^2g_3^2\}.
 \end{align*}
\end{prop}

We have some kind of duality theorems for $\Phi(I)$ and $\mathcal W(I)$.

\begin{prop}\label{fixms6}
 Consider the $\mathbb Z_2$-module of homogeneous quadratic polynomials in $I$ variables $\{z_i\}_{i\in I}$
 $$P_2(I)=\{q\in\mathbb Z_2[I]:q=\sum_{i\in I}a_iz_i^2+
 \sum_{i<j\in I}b_{ij}z_iz_j\} \cong \bigoplus_{i \in I} \mathbb Z_2 \oplus \bigoplus_{i<j \in I} \mathbb Z_2.$$
 Then we have a topological group isomorphism
 $$\Phi(I)
 %\cong\prod_{i\in I}\mathbb Z_2\times\prod_{i<j\in I}\mathbb Z_2
 \cong \mbox{Hom}(P_2(I),\mathbb Z_2) \cong\prod_{i\in I}\mathbb Z_2\times\prod_{i<j\in I}\mathbb Z_2,$$
 with the associated ``perfect pairing''
 $$\langle\_,\_\rangle:\Phi(I)\times P_2(I)\rightarrow\mathbb Z_2.$$
\end{prop}
\begin{proof}
 First of all, note that $P_2(I)$ is generated (as $\mathbb Z_2$-vector space) by the set of monomials $B=\{z_i^2,z_iz_j\}_{i< j\in I}$. In fact, this is a $\mathbb Z_2$-basis of $P_2(I)$. Let $B^*$ be the dual basis
 $$B^*:=\{\varphi_{ij}\}_{i\le j\in I},$$
 where $\varphi_{ij}:P_2(I)\rightarrow\mathbb Z_2$ is given by
 $$\varphi_{ij}(z_kz_l)=\begin{cases}1\mbox{ if }i=k\mbox{ and }j=l \\ 0\mbox{ otherwise}\end{cases}$$
 Now we define a function $\lambda:\{t_i,t_{ij}:i<j\in I\}\rightarrow B^*$ by the rules $t_i\mapsto\varphi_{ii}$ and $t_{ij}\mapsto\varphi_{ij}$. Since $\{t_i,t_{ij}:i<j\in I\}$ is a set of generators of $\Phi(I)$, this function $\lambda$ induces a continuous group homomorphism $\tilde\lambda:\Phi(I)\rightarrow\mbox{Hom}(P_2(I),\mathbb Z_2)$ by the following: let $g=(t_i^{\alpha_i})(t_{ij}^{\beta_{ij}})(1_i)\in\Phi(I)$ (see Proposition \ref{fixms5}). Define $\tilde\lambda(g):P_2(I)\rightarrow\mathbb Z_2$ for $q=\sum_{i\in I}a_iz_i^2+\sum_{i<j\in I}b_{ij}z_iz_j\in P_2(I)$ by
 $$\tilde\lambda(g)(q):=\sum_{i\in I}a_i\alpha_i+\sum_{i<j\in I}b_{ij}\beta_{ij}.$$
 We immediately have that $\tilde\lambda$ is a continuous injective group homomorphism. Since $\lambda$ is bijective and $B^*$ is a $\mathbb Z_2$-basis of $\mbox{Hom}(P_2(I),\mathbb Z_2)$, we have that $\tilde\lambda$ is an isomorphism.
\end{proof}

\begin{prop}\label{fixms7}
 Consider the $\mathbb Z_2$-module of homogeneous linear polynomials in $I$ variables $\{z_i\}_{i\in I}$
 $$P_1(I)=\{q\in\mathbb Z_2[I]:q=\sum_{i\in I}c_iz_i\} \cong \bigoplus_{i \in I} \mathbb Z_2.$$
 Then we have a topological group isomorphism
 $$\mathcal W(I)/\Phi(I)\cong\mbox{Hom}(P_1(I),\mathbb Z_2) \cong\prod_{i\in I}\mathbb Z_2,$$
 with the associated ``perfect pairing''
 $$\langle\_,\_\rangle:\mathcal W(I)/\Phi(I)\times P_1(I)\rightarrow\mathbb Z_2.$$
\end{prop}
\begin{proof}
 First of all, note that $P_1(I)$ is generated (as $\mathbb Z_2$-vector space) by the set of monomials $B=\{z_i\}_{i\in I}$. In fact, this is a $\mathbb Z_2$-basis of $P_1(I)$. Let $B^*$ be the dual basis
 $$B^*:=\{\varphi_{i}\}_{i\in I},$$
 where $\varphi_i:P_1(I)\rightarrow\mathbb Z_2$ is given by
 $$\varphi_i(z_j)=\begin{cases}1\mbox{ if }i=j \\ 0\mbox{ otherwise}\end{cases}$$
 Now we define $\theta:\mathcal W(I)\rightarrow \mbox{Hom}(P_1(I),\mathbb Z_2)$ by the following: for $g=(t_i^{\alpha_i})(t_{ij}^{\beta_{ij}})(x_i^{\gamma_i})\in\mathcal W(I)$, let $\theta(g):P_1(I)\rightarrow\mathbb Z_2$ be the morphism defined by the rule
$$\mbox{For }q=\sum_{i\in I}c_iz_i\in P_1(I),\,\theta(g)(q)=\sum_{i\in I}c_i\gamma_i.$$
Then $\theta$ is a surjective morphism (because $B^*\subseteq\mbox{Im}(\theta)$) and by Proposition \ref{fixms5} $\mbox{Ker}(\theta)=\Phi(I)$. Hence
$$\mathcal W(I)/\Phi(I)=\mathcal W(I)/\mbox{Ker}(\theta)\cong\mbox{Hom}(P_1(I),\mathbb Z_2).$$
\end{proof}

Now, is time to return to our first goal: present the description of $G_F:=\mbox{Gal}_F(F^{(3)})$ by $G_F\cong \mathcal W(I)/\mathcal V(I)$.

\begin{prop}[2.1 in \cite{minac1996witt}]
 $G_F\cong\overline{G^q_F}$.
\end{prop}

Then $G_F$ is a $C$-group on $B$, where $B=\{a_i:i\in I\}$ is an well-ordered basis of $\dot F/\dot F^2$, so, using Theorem \ref{fixms4}, there exists an epimorphism $\pi_B:\mathcal W(B)\rightarrow G_F$. Then we simply take $\mathcal V(B):=\mbox{Ker}(\pi_B)$.

Moreover, J. Minac and M. Spira (again, in \cite{minac1996witt}) gave a  nice explicit description to $\mathcal V (B)$. Let $\mbox{Quat}(F)$ be the subgroup of $Br(F)$, the Brauer group of $F$, generated by the quaternion algebras of $F$.  By Merkurjev's Theorem (\cite{wadsworth55merkurjev}), we have $$\mbox{Quat}(F)\cong \mbox{Br}_2(F)\cong k_2(F),$$
 where $Br_2(F)$ is the subgroup of $Br(F)$ generated by elements of order 2 and  $k_*(F)$ is the graded ring of Milnor's mod 2 reduced K-theory.

  Consider $\varphi_B:P_2(B)\twoheadrightarrow k_2(F)$ as the epimorphism defined by the rule
$$\left(\sum_{i\in I}\alpha_iz_i^2+\sum_{i<j\in I}\beta_{ij}z_iz_j\right)\mapsto\left(\sum_{i\in I}\alpha_il(a_i)l(a_i)+\sum_{i<j\in I}\beta_{ij}l(a_i)l(a_j)\right).$$
Let $Q_B:=\mbox{Ker}(\varphi_B)$.

\begin{fat}[Essentially 2.20 in \cite{minac1996witt}]
 $\mathcal V(B)=Q_B^{\perp}$, where $Q_B^{\perp}=\{v\in\Phi(B):\langle v,Q_B\rangle=0\}$ and $\langle \ , \ \rangle$ is the perfect pairing described in Proposition \ref{fixms6}.
\end{fat}

Let a $F$ be a field with $char(F) \neq 2$. By ``Pontryaguin duality'', let $M_F$ denotes the unique maximal clopen subgroup of $Gal_F= Gal_F(F^{(3)})$ corresponding to $-1 \in SG(F)=\dot{F}/\dot{F}^2$ (Proposition \ref{fixms7}).

\begin{fat}[Essentially 3.3, 3.5, 3.6, 3.7 in \cite{minac1996witt}] Let $F, L \in Field_2$. Then are equivalent:
\begin{enumerate}[a -]
    \item $(W(F), \dot{F}/\dot{F}^2) \cong (W(L), \dot{L}/\dot{L}^2)$ as {\em abstract Witt rings}.
    \item $SG(F) \cong SG(L)$ as {\em special groups}.
    \item $(G_F, M_F) \cong (G_L, M_L)$ as {\em pointed profinite-${\mathcal C}$-groups}.
\end{enumerate}
\end{fat}

Our next step, is use all these facts to obtain a group associated to a (pre)-special group.

\section{The Galois Group of a Pre Special Group}

Lets deal first, with a special group $G$. Let $B=\{a_i:i\in I\}$ be a well ordered $\mathbb Z_2$-basis of $G$ and consider the $\mathcal C$-free group in $B$-generators $\mathcal W(B)$. Define an epimorphism $\pi_B:P_2(B)\rightarrow k_2(G)$ by the rule
$$\left(\sum_{i\in I}\alpha_iz_i^2+\sum_{i<j\in I}\beta_{ij}z_iz_j\right)\mapsto\left(\sum_{i\in I}\alpha_il(a_i)l(a_i)+\sum_{i<j\in I}\beta_{ij}l(a_i)l(a_j)\right)$$
with kernel $Q(B)$. Take $\mathcal V(B):=Q(B)^\perp\subseteq\Phi(B)\subseteq\mathcal W(B)$. Since $\Phi(B)$ is the center of $\mathcal W(B)$ then $\mathcal V(B)\subseteq\mathcal W(B)$ is a (closed) normal subgroup of $\mathcal W(B)$ and we can consider the $\mathcal C$-group $\mathcal W(B)/\mathcal V(B)$.

\begin{defn}[Galois Group - base dependent version]\label{defn:gal1}
 Let $G$ be a special group and $B,\mathcal W(B)$ and $\mathcal V(B)$ as above. We define the \textbf{Galois group of $G$ with respect to $B$} by
$$\mbox{Gal}(G,B):=\mathcal W(B)/\mathcal V(B)$$
\end{defn}

The most essential information of our Galois group is encoded by $Q(B)=\mbox{Ker}(\pi_B)$. We have a useful description by generators that generalizes the one described by J. Minac and M. Spira:

\begin{prop}\label{fixsg1}
 Let $G$ be a special group and $B,\mathcal W(B)$ and $\mathcal V(B)$ as above. Consider a finite subset $B'\subseteq B$, $B'=\{a_{i_0},...,a_{i_{n-1}}\}$ ($i_0<...<i_{n-1}$), and $a,b$ in the linear span of $B'$, say
$$a=\prod_{k<n}a_{i_k}^{\alpha_{i_k}}\mbox{ and }b=\prod_{k<n}a_{i_k}^{\beta_{i_k}},\,\alpha_{i_k},\beta_{i_k}\in\{0,1\}.$$
Consider the polynomial $q^B_{a,b}\in P_2(B)$ given by
$$q^B_{a,b}=\sum_{k<n}\alpha_{i_k}\beta_{i_k}z_{i_k}^2+\sum_{k<l<n}(\alpha_{i_k}\beta_{i_l}+\alpha_{i_l}\beta_{i_k})z_{i_k}z_{i_l}.$$

Note that $q^B_{a,(b_0\cdot...\cdot b_{n-1})}=\sum_{k<n}q^B_{a,b_k}$. Moreover we have the following properties.
\begin{enumerate}[i -]
 \item $\pi_B(q^B_{a,b})=l(a)l(b)\in k_2(G)$.
 \item $q^B_{a,b}$ does not depend on the particular choice of the finite subset $B'\subseteq B$.
 \item $Q(B)$ is generated by $\{q^B_{a,b}:l(a)l(b)=0\}$.
\end{enumerate}
\end{prop}
 \begin{proof}
$ $
\begin{enumerate}[i -]
    \item Note that
    $$l(a)=\sum_{k<n}\alpha_{i_k}l(a_{i_k})\mbox{ and }l(b)=\sum_{k<n}\beta_{i_k}l(a_{i_k}).$$
    Then
    \begin{align*}
        l(a)l(b)&=\left(\sum_{k<n}\alpha_{i_k}l(a_{i_k})\right)
        \left(\sum_{k<n}\beta_{i_k}l(a_{i_k})\right)=
        \sum_{k<n}\sum_{p<n}\alpha_{i_k}\beta_{i_p}l(a_{i_k})l(a_{i_p}) \\
        &=\sum_{k<n}\alpha_{i_k}\beta_{i_k}l(a_{i_k})l(a_{i_k})+
        \sum_{k<n}\sum_{k<p<n}\alpha_{i_k}\beta_{i_p}l(a_{i_k})l(a_{i_p})+\sum_{k<n}\sum_{p<k<n}\alpha_{i_k}\beta_{i_p}l(a_{i_k})l(a_{i_p}) \\
        &=\sum_{k<n}\alpha_{i_k}\beta_{i_k}l(a_{i_k})l(a_{i_k})+
        \sum_{k<n}\sum_{k<p<n}(\alpha_{i_k}\beta_{i_p}+
        \alpha_{i_p}\beta_{i_k})l(a_{i_k})l(a_{i_p})
    \end{align*}
    On the other hand, by definition of $\pi_B$ we get
    \begin{align*}
        \pi_B(q^B_{a,b})&=\sum_{k<n}\alpha_{i_k}\beta_{i_k}l(a_{i_k})l(a_{i_k})+\sum_{k<p<n}(\alpha_{i_k}\beta_{i_p}+\alpha_{i_p}\beta_{i_k})l(a_{i_k})l(a_{i_p}),
    \end{align*}
    completing the proof.

    \item It is an immediate consequence of previous item: if $B_1,B_2$ are finite subsets of $B$ and $a,b$ are elements in the linear span of $B_1$ and in the linear span of $B_2$, then
    $$q^{B_1}_{a,b}=q^{B_2}_{a,b}.$$

    \item Of course, $q^B_{a,b}\in Q(B)$ if and only if $l(a)l(b)=0$ in $k_2(G)$ and hence
    $$[\{q^B_{a,b}:l(a)l(b)=0\}]\subseteq Q(B).$$
    To get the reverse inclusion, let $q=\sum_{k<n}\alpha_{i_k}z^2_{i_k}+\sum_{k,p<n}\beta_{i_ki_p}z_{i_k}z_{i_p}\in Q$. Then
    $$\sum_{k<n}\alpha_{i_k}l(a_{i_k})l(a_{i_k})+\sum_{k,p<n}\beta_{i_ki_p}l(a_{i_k})l(a_{i_p})=0\mbox{ in }k_2(G).$$
    Now, for each $k<n$ let
    $$b_k:=a_{i_k}^{\alpha_{i_k}}\prod_{k<p}a_{i_p}^{
    \beta_{i_k i_p}}.$$
    Then $q=\sum_{k<n}q^B_{a_{i_k},b_{i_k}}$ and
    $$\sum_{k<n}l(a_{i_k})l(b_{i_k})=0\mbox{ in }k_2(G).$$
   We are under the hypothesis of Lemma \ref{fixsg2}. Thus, according Theorem \ref{fixsg3-ktheory}, there exists subsets $\{c_0,...,c_{m-1}\},\{d_0,...,d_{n-1}\}$ of $G$ with $m\ge n$ such that
    \begin{enumerate}
        \item $\{c_0,...,c_{m-1}\}$ is linearly independent and $c_k=a_{i_k}$ for all $k<n$;
        \item $d_k=b_{i_k}$ for all $k<n$ and $d_k=1$ for $k=n,...,m-1$.
        \item For all $x\in[c_0,...,c_{m-1}]$, there is some $r_x\in D_G(1,-x)$ such that for each $k<m$
        $$d_k=\prod_{x\in C_k}r_x$$
        where
        $$C_k=\left\lbrace\prod_{p<m}c_p^{\varepsilon_p}:\varepsilon_p\in\{0,1\}\mbox{ and }\varepsilon_k=1\right\rbrace.$$
    \end{enumerate}
    It follows that
    \begin{align*}
        q &=\sum_{k<n}q^B_{a_{i_k},b_{i_k}}=\sum_{k<m}q^B_{c_k,d_k}
        =\sum_{k<m}q^B_{c_k,\prod_{x\in C_k}r_x} \\
        &=\sum_{k<m}\sum_{x\in C_k}q^B_{c_k,r_x}.
    \end{align*}
    Denoting $C:=[c_0,...,c_{m-1}]$, we have $C=C_0\cup...\cup C_{m-1}$. Then
    \begin{align*}
        q=\sum_{k<m}\sum_{x\in C_k}q^B_{c_k,r_x}
        =\sum_{x\in C}q^B_{x,r_x}.
    \end{align*}
    Since $r_x\in D_G(1,-x)$, we have $l(x)l(r_x)=0$ in $k_2(G)$. Then
    $$q=\sum_{x\in C}q^B_{x,r_x}\in[\{q^B_{a,b}:l(a)l(b)=0\}].$$
\end{enumerate}
\end{proof}

% Examples

Now, we will generalize the Galois group for pre-special groups. The K-theory developed by M. Dickmann and F. Miraglia in \cite{dickmann2006algebraic} is available for pre-special groups. Then we can take the same $B,\mathcal W(B)$ and $\mathcal V(B)$ we are considering until now.

Let $G$ be a special group and $B=\{v_i\}_{i\in I}$, $C=\{w_i\}_{i\in I}$, $D=\{z_i\}_{i\in I}$ be ordered $\mathbb Z_2$-basis of $G$. Then, for all $i\in I$ we have an expression
\begin{align}\label{base-change-prod}
  w_i=\prod_{k\in I}v_k^{m_{ik}},\,m_{ik}\in\{0,1\}\mbox{ for all }i,k\in I,
\end{align}
where the above product has finite support (i.e, $|\{i,k\in I:m_{ik}\ne0\}|<\infty$). In other words, for all $i\in I$, there exist unique sequence in $I$  $i_0 < i_1 < \cdots <  i_{n}$ such that
\begin{align}\label{base-change}
 w_i=v_{i_0}\cdot v_{i_1}\cdot...\cdot v_{i_{n(i)}}.
\end{align}
By abuse of notation, let $C=\{x_i:i\in I\}\subseteq\mathcal W(C)$. We define a function $\mu_{BC}:C\rightarrow\mathcal W(B)$ by the rule
\begin{align}\label{base-change-2}
  x_i\mapsto x_{i_0}\cdot x_{i_1}\cdot...\cdot x_{i_{n}}\mbox{ if }w_i=v_{i_0}\cdot v_{i_1}\cdot...\cdot v_{i_{n}}.
\end{align}
This function is well-defined because both $B$ and $C$ are basis, so the expression \ref{base-change} is unique.

\begin{lem}\label{fixsg2}
Let $G$ be a pre-special group and $B,C,\mu_{BC}$ as above. There is a unique continuous homomorphism $\mu_{BC}:\mathcal W(C)\rightarrow\mathcal W(B)$ that extends $\mu_{BC}$. Also $\mu_{BC}[\Phi(C)]\subseteq\Phi(B)$.
\end{lem}
\begin{proof}
By abuse of notation, let $B=\{x_i:i\in I\}\subseteq\mathcal W(B)$ and $C=\{x_i:i\in I\}\subseteq\mathcal W(C)$. We have $B$ and $C$ as a set of generators converging to 1 in $\mathcal W(B)$ and $\mathcal W(C)$ respectively.

Let $X=\mu_{BC}[C]\subseteq\mathcal W(B)$. Since $\langle X\rangle=\langle B\rangle$ and $\mathcal W(B)=\overline{\langle B\rangle}$, we have that $X$ is a set of generators of $\mathcal W(B)$.

Let $U\subseteq\mathcal W(B)$ be an open subgroup. Since $B$ is a set of generators converging to 1, there is a finite subset $Y=\{x_{i_1},...,x_{j_m}\}\subseteq B$ with $U\cap Y=\emptyset$. Since the set of $\mathbb F_2$-linear combinations of a finite set is finite, there is a finite quantity of elements in $\mu_{BC}[C]$ not belonging to $U$.

 The existence and continuity of $\mu_{BC}$ is an immediate consequence of the Universal Property of $\mathcal W(I)$ (Theorem \ref{fixms4}). An explicit formula for $\mu_{BC}$ is given by the following rule: for $g=(t_i^{\alpha_i})(t_{ij}^{\beta_{ij}})(x_i^{\gamma_i})$ we set
 $$\mu_{BC}(g):=
 \left(t_i^{\left(\alpha_i\sum_{k\in I}m_{ik}\right)}\right)
 \left(t_{ij}^{\left(\beta_{ij}\sum_{r,s\in I}m_{ir}m_{js}\right)}\right)
 \left(x_i^{\left(\gamma_i\sum_{k\in I}m_{ik}\right)}\right).$$
 Since $\Phi(\mathcal W(C))=\mathcal W(C)^2$ and $\Phi(\mathcal W(B))=\mathcal W(B)^2$, we get $\mu_{BC}[\Phi(C)]\subseteq\Phi(B)$.
\end{proof}

\begin{rem} \label{change-rem} A direct calculation show for all $a, b \in G$ that
$$m^{2}_{B,B'}(q^{B}_{ab}) = q^{B'}_{ab}.$$

Denote $\mu^{(2)}_{B,B'} : \Phi(B') \to \Phi(B)$ the restriction of ${\mu}_{B,B'}$ to the Frattini's subgroups and $\mu^{(1)}_{B,B'} : {\mathcal W}(B')/\Phi(B') \to {\mathcal W}(B)/\Phi(B)$ the quotient of $\mu_{B,B'}$. Then, from  the isomorphism in Proposition \ref{fixms6} $\Phi(B) \cong Hom(P_{2}(B),{\mathbb Z}_2)$, we have:
$$\Phi(B') \times P_{2}(B) \to {\mathbb Z}_2 :
<\mu^{1}_{B,B'} -, ->_{B} = <-,m^{1}_{B,B'}->_{B'}$$
so , for all $\sigma' \in \Phi(B')$
$$ <\mu^{2}_{B,B'} (\sigma'), q^{B}_{ab}>_{B} = <\sigma',m^{2}_{B,B'}(q^{B}_{ab})>_{B'} = <\sigma',q^{B'}_{ab}>_{B'}$$ and then $\mu^{2}_{B,B'} [(q^{B'}_{a,b})^{\perp}] = (q^{B}_{ab})^{\perp}$
\end{rem}

\begin{lem}\label{fixsg3}
The morphism $\mu_{BC}$ is an isomorphism, $\mu_{BB}=id_{\mathcal W(B)}$, $\mu_{BC}^{-1}=\mu_{CB}$ and
$$\mu_{BD}=\mu_{BC}\circ\mu_{CD}.$$
\end{lem}
\begin{proof}
 The fact that $\mu_{BB}=id_{\mathcal W(B)}$ and $\mu_{BC}^{-1}=\mu_{CB}$ is an immediate consequence of Lemma \ref{fixsg2}. For the other part, let $B=\{v_i\}_{i\in I}$, $C=\{w_i\}_{i\in I}$ and $D=\{z_i\}_{i\in I}$ be $\mathbb F_2$-basis of $G$. Then for all $i\in I$,
 \begin{align*}
     w_i=\prod_{k\in I}v_k^{m_{ik}},\,
     z_i=\prod_{k\in I}w_k^{n_{ik}},\,
     z_i=\prod_{k\in I}v_k^{p_{ik}},
 \end{align*}
 such that all these products has finite support. Then
 \begin{align*}
     z_i&=\prod_{k\in I}w_k^{n_{ik}}
     =\prod_{k\in I}\left(\prod_{r\in I}v_r^{m_{kr}}\right)^{n_{ik}}
     =\prod_{k\in I}\prod_{r\in I}v_r^{n_{ik}m_{kr}}
     =\prod_{r\in I}\prod_{k\in I}v_r^{n_{ik}m_{kr}}
     =\prod_{r\in I}v_r^{p_{ir}}.
 \end{align*}
  Moreover
 $$\sum_{k\in I}\sum_{r\in I}n_{ik}m_{kr}=
 \sum_{r\in I}\sum_{k\in I}n_{ik}m_{kr}=\sum_{r\in I}p_{ir}.$$
 Then for all $g=(t_i^{\alpha_i})(t_{ij}^{\beta_{ij}})(x_i^{\gamma_i})\in\mathcal W(D)$,
 \begin{align*}
     &\mu_{BC}\circ\mu_{CD}(g)=\mu_{BC}(\mu_{CD}(g))= \\
     &\mu_{BC}\left(
     \left(t_i^{\alpha_i\left(\sum_{k\in I}n_{ik}\right)}\right)
    \left(t_{ij}^{\left(\beta_{ij}\sum_{r,s\in I}n_{ir}n_{js}\right)}\right)
    \left(x_i^{\left(\gamma_i\sum_{k\in I}n_{ik}\right)}\right)
     \right)=\\
     &\left(t_i^{\alpha_i\left(\sum_{k\in I}n_{ik}\right)\left(\sum_{r\in I}m_{kr}\right)}\right)
 \left(t_{ij}^{\beta_{ij}\left(\sum_{r,s\in I}n_{ir}n_{js}\right)\left(\sum_{e,f\in I}m_{re}m_{sf}\right)}\right)
 \left(x_i^{\gamma_i\left(\sum_{k\in I}n_{ik}\right)\left(\sum_{r\in I}m_{kr}\right)}\right)=\\
 &\left(t_i^{\alpha_i
 \left(\sum_{r\in I}\sum_{k\in I}n_{ik}m_{kr}\right)}\right)
 \left(t_{ij}^{\beta_{ij}
 \left(\sum_{r,s\in I}\sum_{e,f\in I}(n_{ir}m_{re})(n_{js}m_{sf})\right)}\right)
 \left(x_i^{\gamma_i
 \left(\sum_{r\in I}\sum_{k\in I}n_{ik}m_{kr}\right)}\right)=\\
 &\left(t_i^{\alpha_i\left(\sum_{r\in I}p_{ir}\right)}\right)
    \left(t_{ij}^{\beta_{ij}\left(\sum_{e,f\in I}p_{ie}p_{jf}\right)}\right)
    \left(x_i^{\gamma_i\left(\sum_{r\in I}p_{ir}\right)}\right)
    =\mu_{BD}(g).
 \end{align*}
 Then we get $\mu_{BD}=\mu_{BC}\circ\mu_{CD}$.
\end{proof}

Now, consider the Equation \ref{base-change-prod}. This expression induce isomorphisms $m^1_{BC}:P_1(B)\rightarrow P_1(C)$ and $m^2_{BC}:P_2(B)\rightarrow P_2(C)$ given by the rules
\begin{align*}
  m^1_{BC}\left(\sum_{i\in I}c_iz_i\right)
  &:=\sum_{i\in I}c_i\left(\sum_{k\in I}m_{ik}\right)z_i  \\
  m^2_{BC}\left(\sum_{i\in I}a_iz_i^2+\sum_{i<j\in I}b_{ij}z_iz_j\right)
  &:=
  \sum_{i\in I}a_i\left(\sum_{k\in I}m_{ik}\right)z^2_i+
  \sum_{i<j\in I}b_{ij}\left[
  \left(\sum_{r\in I}m_{ir}\right)\left(\sum_{s\in I}m_{js}\right)\right]z_iz_j.
\end{align*}
where all these sums has finite support.

\begin{lem}\label{fixsg4}
We have
$$m^1_{BD}=m^1_{BC}\circ m^1_{CD}\mbox{ and }
m^2_{BD}=m^2_{BC}\circ m^2_{CD}.$$
\end{lem}
\begin{proof}
Lets recover the calculations in the proof of Lemma \ref{fixsg3}: let $B=\{v_i\}_{i\in I}$, $C=\{w_i\}_{i\in I}$ and $D=\{z_i\}_{i\in I}$ be $\mathbb F_2$-basis of $G$. Then for all $i\in I$,
  \begin{align*}
     w_i=\prod_{k\in I}v_k^{m_{ik}},\,
     z_i=\prod_{k\in I}w_k^{n_{ik}},\,
     z_i=\prod_{k\in I}v_k^{p_{ik}},
 \end{align*}
 such that all these products has finite support. Then
 \begin{align*}
     z_i&=\prod_{k\in I}w_k^{n_{ik}}
     =\prod_{k\in I}\left(\prod_{r\in I}v_r^{m_{kr}}\right)^{n_{ik}}
     =\prod_{k\in I}\prod_{r\in I}v_r^{n_{ik}m_{kr}}
     =\prod_{r\in I}\prod_{k\in I}v_r^{n_{ik}m_{kr}}
     =\prod_{r\in I}v_r^{p_{ir}}.
 \end{align*}
 Moreover
 $$\sum_{k\in I}\sum_{r\in I}n_{ik}m_{kr}=
 \sum_{r\in I}\sum_{k\in I}n_{ik}m_{kr}=\sum_{r\in I}p_{ir}.$$
 Then
 \begin{align*}
     m^1_{BC}\left(m^1_{CD}\left(\sum_{k\in I}c_kz_i\right)\right)
     &=m^1_{BC}\left(\sum_{i\in I}c_i\left(\sum_{k\in I}n_{ik}\right)z_i\right)
     =\sum_{i\in I}c_i\left(\sum_{k\in I}n_{ik}\right)\left(\sum_{r\in I}m_{kr}\right)z_i \\
     &=\left(\sum_{r\in I}p_{ir}\right)z_i
     =m^1_{BD}\left(\sum_{k\in I}c_kz_i\right)
 \end{align*}
 and hence $m^1_{BD}=m^1_{BC}\circ m^1_{CD}$. In the same reasoning,
 \begin{align*}
     &m^2_{BC}\left(m^2_{CD}\left(\sum_{i\in I}a_iz_i^2+\sum_{i<j\in I}b_{ij}z_iz_j\right)\right)= \\
     &m^2_{BC}\left(\sum_{i\in I}a_i\left(\sum_{k\in I}n_{ik}\right)z^2_i+
  \sum_{i<j\in I}b_{ij}\left[
  \left(\sum_{r\in I}n_{ir}\right)\left(\sum_{s\in I}n_{js}\right)\right]z_iz_j\right)= \\
  &\sum_{i\in I}a_i\left(m^2_{BC}\left(\left(\sum_{k\in I}n_{ik}\right)z^2_i\right)\right)+
  m^2_{BC}\left(\sum_{i<j\in I}b_{ij}\left[\left(
  \sum_{r\in I}n_{ir}\right)\left(\sum_{s\in I}n_{js}\right)\right]z_iz_j\right)= \\
  &\sum_{i\in I}a_i\left(\sum_{k\in I}n_{ik}\left(\sum_{r\in I}m_{kr}\right)z^2_i\right)+
  \sum_{i<j\in I}b_{ij}
  \left[
  \left(\sum_{r\in I}n_{ir}\right)\left(\sum_{s\in I}n_{js}\right)
  \left(\sum_{e\in I}m_{re}\right)\left(\sum_{f\in I}m_{sf}\right)
  \right]z_iz_j= \\
  &\sum_{i\in I}a_i\left(
  \left(\sum_{r\in I}\sum_{k\in I}n_{ik}m_{kr}\right)z^2_i\right)+
  \sum_{i<j\in I}b_{ij}
  \left[
  \left(\sum_{e\in I}\sum_{r\in I}n_{ir}m_{re}\right)
  \left(\sum_{f\in I}\sum_{s\in I}n_{js}m_{sf}\right)
  \right]z_iz_j=\\
  &\sum_{i\in I}a_i\left(\sum_{r\in I}p_{ir}\right)z^2_i+
  \sum_{i<j\in I}b_{ij}\left[
  \left(\sum_{e\in I}p_{ie}\right)
  \left(\sum_{f\in I}p_{jf}\right)
 \right] z_iz_j=\\
  &m^2_{BD}\left(\sum_{i\in I}a_iz_i^2+\sum_{i<j\in I}b_{ij}z_iz_j\right)
 \end{align*}
 provide that $m^2_{BD}=m^2_{BC}\circ m^2_{CD}$.
\end{proof}

Note that $m^1_{BC},m^2_{BC}$ induces respectively the isomorphisms
\begin{align*}
 \mathfrak m^1_{BC}&:\mbox{Hom}(P_1(C),\mathbb Z_2)\rightarrow
\mbox{Hom}(P_1(B),\mathbb Z_2) \\
\mathfrak m^2_{BC}&:\mbox{Hom}(P_2(C),\mathbb Z_2)\rightarrow
\mbox{Hom}(P_2(B),\mathbb Z_2).
\end{align*}
given by the respective rules: if $f:P_1(C)\rightarrow\mathbb Z_2$ and $q=\sum_{i\in I}c_iz_i\in P_1(B)$ then
\begin{align*}
 \mathfrak m^1_{BC}(f)(q)&:=f(m^1_{BC}(q))=f\left(\sum_{i\in I}c_i\left(\sum_{k\in I}m_{ik}\right)z_i\right).
\end{align*}
In the same reasoning, if $f:P_2(C)\rightarrow\mathbb Z_2$ and $q=\sum_{i\in I}a_iz_i^2+\sum_{i<j\in I}b_{ij}z_iz_j\in P_2(B)$ then
\begin{align*}
 \mathfrak m^2_{BC}(f)(q)&:=f(m^2_{BC}(q))
 =f\left(
 \sum_{i\in I}a_i\left(\sum_{k\in I}m_{ik}\right)z^2_i+
  \sum_{i<j\in I}b_{ij}\left[
  \left(\sum_{r\in I}m_{ir}\right)\left(\sum_{s\in I}m_{js}\right)
  \right]z_iz_j\right)
\end{align*}

Now denote
  $$\mu^1_{BC}:\mathcal W(C)/\Phi(C)\rightarrow\mathcal W(B)/\Phi(B)$$
  the quotient of $\mu_{BC}$ and
  $$\mu^2_{BC}:\Phi(C)\rightarrow\Phi(B)$$
  the restriction of $\mu_{BC}$ to the Frattini's subgroups. Also consider the isomorphisms
  \begin{align*}
  \tilde\theta:\mathcal W(I)/\Phi(I)&\xrightarrow{\cong}\mbox{Hom}(P_1(I),\mathbb Z_2) \\
 \tilde\lambda:\Phi(I)&\xrightarrow{\cong}\mbox{Hom}(P_2(I),\mathbb Z_2)
  \end{align*}
 of Propositions \ref{fixms6}, and \ref{fixms7}.

\begin{lem}\label{fixsg5}
Denote $\pi_B:\mathcal W(B)\rightarrow\mathcal W(B)/\Phi(B)$ the canonical projection, with the same for $\pi_C$. Then we have a commutative diagram
$$
\xymatrix@!=4.5pc{
\mathcal W(C)\ar[d]_{\mu_{BC}}\ar[r]^{\theta_C} &
**[r]\mbox{Hom}(P_1(C),\mathbb Z_2)\ar[d]^{\mathfrak m^1_{BC}} \\
\mathcal W(B)\ar[r]^{\theta_B} &
**[r]\mbox{Hom}(P_1(B),\mathbb Z_2) }
$$
which induces a commutative diagram
$$
\xymatrix@!=4pc{
\mathcal W(C)\ar[d]_{\pi_C}\ar[dr]^{\theta_C} & \\
**[l]\mathcal W(C)/\Phi(C)\ar[r]^{\tilde\theta_C}\ar[d]_{\mu^1_{BC}} &
**[r]\mbox{Hom}(P_1(C),\mathbb Z_2)\ar[d]^{\mathfrak m^1_{BC}} \\
**[l]\mathcal W(B)/\Phi(B)\ar[r]^{\tilde\theta_B} &
**[r]\mbox{Hom}(P_1(B),\mathbb Z_2)\\
\mathcal W(B)\ar[u]^{\pi_B}\ar[ur]_{\theta_B} & }
$$
\end{lem}
\begin{proof}
Let $g=(t_i^{\alpha_i})(t_{ij}^{\beta_{ij}})(x_i^{\gamma_i})\in\mathcal W(C)$ and $q=\sum_{i\in I}c_iz_i\in P_1(B)$. Then
\begin{align*}
    (\mathfrak m^1_{BC}\circ\theta)(g)(q)&=
    \mathfrak m^1_{BC}(\theta(g)(q))=\theta(g)(m^1_{BC}(q)) \\
    &=\theta(g)\left(
    \sum_{i\in I}c_i\left(\sum_{k\in I}m_{ik}\right)z_i\right)
    =\sum_{i\in I}c_i\left(\sum_{k\in I}m_{ik}\right)\gamma_i\\
    &=\sum_{i\in I}\sum_{k\in I}c_im_{ik}\gamma_i.
\end{align*}
On the other hand,
\begin{align*}
(\theta\circ\mu_{BC})(g)(q)&=\theta(\mu_{BC}(g))(q) \\
&=\theta\left(
\left(t_i^{\alpha_i\left(\sum_{k\in I}m_{ik}\right)}\right)
 \left(t_{ij}^{\beta_{ij}\left(\sum_{r,s\in I}m_{ir}m_{js}\right)}\right)
 \left(x_i^{\gamma_i\left(\sum_{k\in I}m_{ik}\right)}\right)
\right)\left(\sum_{i\in I}c_iz_i\right) \\
&=\sum_{i\in I}c_i\left(\gamma_i\left(\sum_{k\in I}m_{ik}\right)\right)
=\sum_{i\in I}\sum_{k\in I}c_im_{ik}\gamma_i=(\mathfrak m^1_{BC}\circ\theta)(g)(q).
\end{align*}
Then $\mathfrak m^1_{BC}\circ\theta=\theta\circ\mu_{BC}$. Since $\theta_B=\tilde\theta_B\circ\pi_B$ and $\theta_C=\tilde\theta_C\circ\pi_C$ we have the desired commutative diagram.
\end{proof}

\begin{lem}\label{fixsg5b}
We have a commutative diagram

$$
\xymatrix@!=4.5pc{**[l]\Phi(C)\ar[r]^{\tilde\lambda_C}\ar[d]_{\mu^2_{BC}} &  **[r]\mbox{Hom}(P_2(C),\mathbb Z_2)\ar[d]^{\mathfrak m^2_{BC}} \\
**[l]\Phi(B)\ar[r]_{\tilde\lambda_B} & **[r]\mbox{Hom}(P_2(B),\mathbb Z_2)}
$$
\end{lem}
\begin{proof}
Let $g=(t_i^{\alpha_i})(t_{ij}^{\beta_{ij}})(1_i)\in\Phi(C)$ and
$\sum_{i\in I}a_iz_i^2+\sum_{i<j\in I}b_{ij}z_iz_j\in P_2(B)$. We have
\begin{align*}
    &(\mathfrak m^2_{BC}\circ\tilde\lambda_C)(g)(q)=
    \tilde\lambda_C(g)(m^2_{BC}(q))=\\
    &\tilde\lambda_C(g)
    \left(
    \sum_{i\in I}a_i\left(\sum_{k\in I}m_{ik}\right)z^2_i+
  \sum_{i<j\in I}b_{ij}\left[
  \left(\sum_{r\in I}m_{ir}\right)\left(\sum_{s\in I}m_{js}\right)\right]z_iz_j
    \right)=\\
    &\sum_{i\in I}a_i\left(\sum_{k\in I}m_{ik}\right)\alpha_i+
  \sum_{i<j\in I}b_{ij}\left[
  \left(\sum_{r\in I}m_{ir}\right)\left(\sum_{s\in I}m_{js}\right)\right]\beta_{ij}
\end{align*}
On the other hand,
\begin{align*}
    &(\tilde\lambda_B\circ\mu^2_{BC})(g)(q)=\tilde\lambda(\mu^2_{BC}(g)(q)) \\
    &=\tilde\lambda_B\left(
    \left(t_i^{\alpha_i\left(\sum_{k\in I}m_{ik}\right)}\right)
 \left(t_{ij}^{\beta_{ij}\left(\sum_{r,s\in I}m_{ir}m_{js}\right)}\right)
 (1_i)\right)\left(\sum_{i\in I}a_iz_i^2+\sum_{i<j\in I}b_{ij}z_iz_j\right) =\\
 &\sum_{i\in I}a_i\left(\alpha_i\left(\sum_{k\in I}m_{ik}\right)\right)+
 \sum_{i<j\in I}b_{ij}\left(\beta_{ij}\left(\sum_{r,s\in I}m_{ir}m_{js}\right)\right)
\end{align*}
proving that $\mathfrak m^2_{BC}\circ\tilde\lambda_C=\tilde\lambda_B\circ\mu^2_{BC}$.
\end{proof}

\begin{lem}\label{fixsg6}
 With the notations of Lemmas \ref{fixsg1}-\ref{fixsg5} we have the following.
 \begin{enumerate}[i -]
  \item The arrows $\mu^1_{BC}$ and $\mu^2_{BC}$ are isomorphisms. Moreover, for all well-ordered basis $B,C,D$ we have $\mu^1_{BB}=id$, $\mu^2_{BB}=id$, $\mu^1_{BD}=\mu^1_{BC}\circ\mu^1_{CD}$ and $\mu^2_{BD}=\mu^2_{BC}\circ\mu^2_{CD}$.

  \item The isomorphism $\mu_{BC}:\mathcal W(C)\rightarrow\mathcal W(B)$ restricts to an isomorphism $\mathcal V(C)\rightarrow\mathcal V(B)$ so we get quotient isomorphism
$$\tilde\mu_{BC}:\mathcal W(C)/\mathcal V(C)\rightarrow\mathcal W(B)/\mathcal V(B).$$
  \item If $B,C,D$ are well-ordered base of $G$, then $\tilde\mu_{CC}=id$ and $\tilde\mu_{BD}=\tilde\mu_{BC}\circ\tilde\mu_{CD}$.
 \end{enumerate}
\end{lem}
\begin{proof}
 $ $
 \begin{enumerate}[i -]
     \item Just use the same calculations made in Lemma \ref{fixsg3}.

     \item By Proposition \ref{fixsg1} we have
     $$\mathfrak m^2_{BC}(Q(B))=Q(C).$$
     Since $\mathcal V(B)=Q(B)^\perp$, we have an induced isomorphism
     $\mu_{BC}|_{\mathcal V(B)}:\mathcal V(B)\xrightarrow{\cong}\mathcal V(C)$, legitimating the quotient isomorphism
     $$\tilde\mu_{BC}:\mathcal W(C)/\mathcal V(C)\rightarrow\mathcal W(B)/\mathcal V(B).$$

     \item It is an immediate consequence of item (i).
 \end{enumerate}
\end{proof}

\begin{defn}[Galois Group - base independent version]\label{defn:gal2}
 Let $G$ be a pre-special group. Take
 $$E_G=\{B:B\mbox{ is a well-ordered }\mathbb F_2\mbox{-basis of }G\}.$$
 Consider the set $E_G$ endowed with the trivial groupoid operation of concatenation of pairs (i.e., the arrows are $E_G \times E_G$) and take the functor $\mbox{Gal}:E_G\rightarrow\mathcal C$, (where $\mathcal C$ is the category of $\mathcal C$-groups) given by the following rules: for an object $B\in E_G$, $\mbox{Gal}(B):=\mathcal W(B)/\mathcal V(B)$ and for an arrow $(B,C) \in E^2_G$,
$$\mbox{Gal}(B,C)=\mu_{BC}:\mathcal W(C)/\mathcal V(C)\rightarrow\mathcal W(B)/\mathcal V(B).$$
We define the \textbf{Galois group } of $G$, notation $\mbox{Gal}(G)$ by
$$\mbox{Gal}(G):=\varprojlim_{B\in E_G}\mathcal W(B)/\mathcal V(B).$$
\end{defn}

\begin{rem} Keeping the notation above, note that
$$\pi_B : \mbox{Gal}(G):=\varprojlim_{B\in E_G}\mathcal W(B)/\mathcal V(B) \to \mathcal W(B)/\mathcal V(B)$$ is an isomorphism, for each $B \in E_G$. This holds because $$\mbox{Gal}(B,C)=\mu_{BC}:\mathcal W(C)/\mathcal V(C)\rightarrow\mathcal W(B)/\mathcal V(B)$$
is an isomorphism for each  arrow $(B,C) \in E^2_G$.
\end{rem}

\section{On the structure of Galois Groups  of Pre Special Groups}

%{\textcolor{red}{Encrespou a partir daqui}}

As occurs with fields, the Galois group of a pre-special groups (in a certain subclass) is able to encode many relevant quadratic information.

{\em All pre-special groups occurring is this section will be assumed $k$-stable.}

We start this Section, by providing more details on the structure of $\mathcal C$-groups.

Let $\mathcal G$ be a $\mathcal C$-group on $I$-minimal generators. By Theorem \ref{fixms4}, There is an epimorphism $\lambda:\mathcal W(I)\rightarrow\mathcal G$ with kernel $\mathcal V\subseteq\Phi(I)$, and then, we have an isomorphism $\tilde\lambda:\mathcal W(I)/\mathcal V\rightarrow\mathcal G$.

\begin{prop} \label{fixhugo2}
With the above notation, we have the following.
\begin{enumerate}[i -]
    \item We have a natural bijection
    \begin{align*}
      &\{M\subseteq\mathcal G:M\mbox{ is a maximal open subgroup}\}\cong \\
    &\{M\subseteq\mathcal W(I):M\mbox{ is a maximal open subgroup}\}.
    \end{align*}
    Then $\Phi(\mathcal G)\cong \Phi(I)/\mathcal V$ and $\mathcal G/\Phi(\mathcal G)\cong\mathcal W(I)/\Phi(I)$.

    \item The maximal closed subgroups of $\mathcal W(I)$ are precisely the clopen (normal) subgroups with quotient $\mathbb Z_2$. Moreover, we have a natural bijection
    $$\{M\subseteq\mathcal G:M\mbox{ is a maximal open subgroup}\}\cong P_{fin}(I)\setminus\{\emptyset\}.$$
    and that extends to a natural bijection
    $$\{N\subseteq\mathcal G:N\mbox{ is an open subgroup with index} \leq 2\}\cong P_{fin}(I).$$

    \item We have a natural isomorphism of $\mathbb Z_2$-modules
    $$ Homcont(\mathcal G, \mathbb Z_2) \cong fsFunc(I, \mathbb Z_2)$$
    where $fsFunc(I, \mathbb Z_2)$ is the set of all function $f: I \to \mathbb Z_2$ with finite support.
\end{enumerate}
\end{prop}
\begin{proof}
 $ $
 \begin{enumerate}[i -]
     \item The desired bijection follows by the bijection
     \begin{align*}
      &\{M\subseteq\mathcal W(I)/\mathcal V:M\mbox{ is a maximal open subgroup}\}\cong \\
    &\{M\subseteq\mathcal W(I):M\mbox{ is a maximal open subgroup}\}.
    \end{align*}
    given by the following rule: lets $q:\mathcal W(I)\rightarrow\mathcal W(I)/\mathcal V$ denote the canonical projection. We have a function $\overline q:\mathcal P(\mathcal W(I)/\mathcal V)\rightarrow\mathcal P(\mathcal W(I))$ given by the rule
    $$\overline q(X):=q^{-1}[X]\mbox{ (the inverse image)}.$$
    This function $\overline{q}$ induces the desired bijection. Since the Frattini subgroup of $\mathcal G$ is the intersection of all open normal subgroups we have (via $\tilde\lambda$ and the bijection) $\Phi(\mathcal G)\cong\Phi(I)/\mathcal V$. Then
    $$\mathcal G/\phi(\mathcal G)\cong(\mathcal W(I)/\mathcal V)/(\Phi(I)/\mathcal V)\cong\mathcal W(I)/\Phi(I).$$

     \item For $\{i_0,...,i_n\}\subseteq I$ denote
     $$\zeta_I(i_0,...,i_n):=\{\sigma\in\mathcal W(I):
     \gamma_{i_0}(\sigma)+...+\gamma_{i_n}(\sigma)=0\}.$$
     We have that $\zeta_I(i_0,...,i_n)$ is a subgroup of $\mathcal W(I)$. Now let $\tau,\theta\in\mathcal W(I)\setminus\zeta_I(i_0,...,i_n)$. Then for all $i\in I$,
    $$\gamma_i(\theta^{-1}\tau)=\gamma_i(\theta)+\gamma_i(\tau).$$
     Therefore
     $$\sum^n_{p=1}\gamma_{i_p}(\theta^{-1}\tau)=
     \sum^n_{p=1}[\gamma_{i_p}(\theta)+\gamma_{i_p}(\tau)]=
     \sum^n_{p=1}\gamma_{i_p}(\theta)+\sum^n_{p=1}\gamma_{i_p}(\tau)=
     1+1=0.$$
     Then $\theta^{-1}\tau\in\zeta_I(i_0,...,i_n)$ which imply
     $$\mathcal W(I)/\zeta_I(i_0,...,i_n)=\{\overline1,\overline\tau\}\cong\mathbb Z_2.$$

     To verify that $\zeta_I(i_0,...,i_n)$ is clopen, note that
     $$\gamma_{i_0}(\sigma)+...+\gamma_{i_n}(\sigma)=0\mbox{ iff }
     \begin{cases}
     \gamma_{i_0}(\sigma)+...+\gamma_{i_{n-1}}(\sigma)=0\mbox{ and }\gamma_{i_n}(\sigma)=0 \mbox{ or }\\
     \gamma_{i_0}(\sigma)+...+\gamma_{i_{n-1}}(\sigma)=1\mbox{ and }\gamma_{i_n}(\sigma)=1.
     \end{cases}$$
     In other words,
     $$\zeta_I(i_0,...,i_n)=[\zeta_I(i_0,...,i_{n-1})\cap M_{i_n}]\cup
     [\zeta_I(i_0,...,i_{n-1})^c\cap M_{i_n}^c].$$
     So, in order to verify that $\zeta_I(i_0,...,i_n)$ is clopen is enough to deal with the case $n=0$. But $\zeta_I(i_0)=M_{i_0}$ is in fact a clopen, which provide (by induction) that $\zeta_I(i_0,...,i_n)$ is clopen for all $i_0,...,i_n\in I$. Then $\zeta_I(i_0,...,i_n)$ is a maximal clopen subgroup of $\mathcal W(I)$, and we have a well-defined injective function
     $$\zeta_I:\mathcal P_{fin}(I)\setminus\{\emptyset\}\rightarrow
     \{M\subseteq\mathcal W(I):M\mbox{ is a maximal open subgroup}\}$$
     given by the rule $\{i_0,...,i_n\}\mapsto\zeta_I(i_0,...,i_n)$.

     For surjectivity, let $M$ be a maximal open subgroup. Then $M$ is closed of finite index and by Lemma \ref{rz2.8.7}(a), $M$ has index 2 in $G$. Using Propositions \ref{fixms2} and \ref{fixms3} and the compacity of $\mathcal W(I)$, there exists $i_1,...,i_n,j_1,...,j_m,k_1,...,k_p\in I$ with
     $$M_{i_1}\cap...\cap M_{i_n}\cap S_{j_1}\cap...\cap S_{j_m}\cap D_{k_1}\cap...\cap D_{k_p}\subseteq M.$$
     Note that we have
     $$M_{i_1}\cap...\cap M_{i_n}\cap S_{j_1}\cap...\cap S_{j_m}\cap D_{k_1}\cap...\cap D_{k_p}\subseteq     \zeta_I(i_1,...,i_n,j_1,...,j_m,k_1,...,k_p).$$
     Lets denote $\zeta_I(i_1,...,i_n,j_1,...,j_m,k_1,...,k_p):=\zeta_I(\vec i,\vec j,\vec k)$ and $H:=M\cap \zeta_I(\vec i,\vec j,\vec k)$. Suppose $H\ne M$ and let $\tau,\theta\in M\setminus H$. The same calculations made for injectivity shows that $\theta^{-1}\tau\in\zeta_I(\vec i,\vec j,\vec k)$ which imply
     $$M/H=\{\overline1,\overline\tau\}\cong\mathbb Z_2.$$
     Moreover, using the same calculations made in Proposition \ref{fixms2}(b) we have that $H$ has index $\mathbb Z_4$ in $\mathcal W(I)$. Since $H\subseteq M$ and $H\subseteq\zeta_I(\vec i,\vec j,\vec k)$ with both maximal clopen subgroups, by Lemma \ref{ZqDq-le}(i) we have $M=\zeta_I(\vec i,\vec j,\vec k)$, contradicting the assumption $H\ne M$. Then $H=M$ and we have $M=\zeta_I(\vec i,\vec j,\vec k)$. Therefore we have bijections
     \begin{align*}
         \mathcal P_{fin}(I) \setminus\{\emptyset\}&\cong
         \{M\subseteq\mathcal W(I):M\mbox{ is a maximal open subgroup}\} \\
         &\cong\{M\subseteq\mathcal G:M\mbox{ is a maximal open subgroup}\}.
     \end{align*}
     Therefore
     $$\{N\subseteq\mathcal G:N\mbox{ is an open subgroup with index} \leq 2\}\cong P_{fin}(I).$$

     \item Since $\mathcal G = \mathcal W(I)/{\mathcal V}$ and $${\mathcal V} \subseteq \Phi(I) = \bigcap \{ker(\varphi) : \varphi \in  Homcont(\mathcal W(I), \mathbb Z_2)\},$$
     then the natural epimorphism $\mathcal W(I) \twoheadrightarrow \mathcal W(I)/{\mathcal V}$ induces the isomorphism $$Homcont(\mathcal G, \mathbb Z_2) \cong  Homcont(\mathcal W(I), \mathbb Z_2).$$
     Since $\mathbb Z_2$ is a finite/discrete $\mathcal C$-group, then the universal property of $\mathcal W(I)$ (Theorem \ref{fixms4}) gives a natural isomorphism $Homcont(\mathcal W(I), \mathbb Z_2) \cong fsFunc(I, \mathbb Z_2)$.

     Alternatively, the result follows also from the item (ii) above and the (obvious) natural bijections
     $$Homcont(\mathcal G, \mathbb Z_2) \cong \{N\subseteq\mathcal G:N\mbox{ is an open subgroup with index} \leq 2\}$$
     $$P_{fin}(I) \cong fsFunc(I, \mathbb Z_2).$$
 \end{enumerate}
\end{proof}

%{\textcolor{red}{Aqui come\c cou a faltar os enunciados}}
\begin{prop} [Prontryagin duality] \label{fixhugo3}
Let $\mathcal G=\mbox{Gal}(G)$ for some pre-special group $G$. Then
\begin{enumerate}[i -]
    \item There is a canonical bijection
    $$\mathbb{M}: G\overset\cong\to\{M\subseteq\mathcal G:M\mbox{ is a (closed, normal) subgroup of index less or equal to }2\}$$
    $a \mapsto M_a$ such that it induces a canonical bijection
    $$G\setminus\{1\}\cong\{M\subseteq\mathcal G:M\mbox{ is a maximal subgroup}\}.$$

    \item There is a canonical isomorphism of $\mathbb Z_2$-modules $\psi_G : G \overset\cong\to  Homcont(\mathcal G, \mathbb Z_2)$, $a \mapsto \mu_a$, where $\mu_a : \mathcal G \to \mathbb Z_2$ is the unique continuous homomorphism such that $ker(\mu_a) = M_a$.

 \item There is a canonical isomorphism of pro-2-groups $\phi_G:\mathcal G/\Phi(\mathcal G)\rightarrow\mbox{Hom}(G,\mathbb Z_2)$.
\end{enumerate}
\end{prop}
\begin{proof}
 $ $
 \begin{enumerate}[i -]
     \item This follows directly from Proposition \ref{fixhugo2}, since $\pi_B : Gal(G) \overset\cong\to {\mathcal W}(B)/{\mathcal V}(B)$ and ${\mathcal V}(B) \subseteq \Phi(B)$, for every well ordered basis $B$ of $G$.

     \item By Proposition \ref{fixhugo2}(iii), for each well ordered basis $B$ in $G$, we have a natural isomorphism of $\mathbb Z_2$-modules.
     $$ Homcont(\mathcal W(B)/\mathcal V(B), \mathbb Z_2) \cong fsFunc(B, \mathbb Z_2)$$
     This is, in fact, an isomorphism of $\mathbb Z_2$-modules. Taking into account the isomorphisms of "change of base", we glue the above isomorphisms to obtain the natural isomorphism
     $$Homcont(Gal(G), \mathbb Z_2) = Homcont(\varprojlim_{B\in E_G}\mathcal W(B)/\mathcal V(B), \mathbb Z_2) \cong  \varprojlim_{B\in E_G} fSFunc(B, \mathbb Z_2) \cong G$$

\item Note that $\pi_{\mathcal G} : {\mathcal G} \twoheadrightarrow {\mathcal G}/\Phi({\mathcal G})$ induces a $\mathbb Z_2$-isomorphism
$$\pi^*_{\mathcal G} : Homcont(\mathcal G/\Phi({\mathcal G}), \mathbb Z_2) \overset\cong\to Homcont(\mathcal G, \mathbb Z_2).$$

By Lemma \ref{fixsg5}, the isomorphisms described in Proposition \ref{fixms7}, namely $$\mathcal W(B)/\Phi(B)\cong\mbox{Hom}(P_1(B),\mathbb Z_2),$$
are natural. Thus we obtain a natural isomorphism
$$\mathcal G/\Phi({\mathcal G}) \cong  Hom(G, \mathbb Z_2)$$
 \end{enumerate}
\end{proof}

\begin{rem}
Note that combining items (iii), (ii) of the Proposition above, we obtain the  Pontryagin duality:
$$ \mathcal G/\Phi({\mathcal G}) \cong  Hom(G, \mathbb Z_2) \cong Hom(Homcont(\mathcal G/\Phi({\mathcal G}), \mathbb Z_2), \mathbb Z_2)$$

$$ G\cong  Homcont(\mathcal G/\Phi({\mathcal G}), \mathbb Z_2) \cong Homcont(Hom(G, \mathbb Z_2), \mathbb Z_2)$$

This induces a canonical duality between the pointed $\mathbb Z_2$-module $(G, -1)$  and the "pointed" pro-2-group $(Gal(G), M)$, where $M \subseteq Gal(G)$ is an open subgroup of index $\leq 2$.
\end{rem}

Let $G$ be a k-stable special group. Write $\mathcal G=\mbox{Gal}(G)$.

The isomorphism of pro-2-groups $\phi_G:\mathcal G/\Phi(\mathcal G)\rightarrow\mbox{Hom}(G,\mathbb Z_2)$, determines a "perfect pairing" $\hat\phi_G:\mathcal G/\Phi(\mathcal G)\times G\rightarrow\mathbb Z_2$ given by the rule
$$\hat\phi_G(\overline\alpha,g):=\langle\overline\alpha,g\rangle:=\phi_G(\overline\alpha)(g).$$
We will denote  $( \ )^\perp$, generically, both the correspondences between subsets of $\mathcal G/\Phi(\mathcal G)$ and subsets of $G$.

%    This "perfect pairing"  gives an anti-isomorphism, , of complete lattices between the poset of closed subgroups of $\mathcal G/\Phi(\mathcal G)$ and the poset of subgroups of $G$. Moreover,  this map restricts to a bijection
% between discrete subgroups of order 2
% $\{id, \sigma/\Phi(\mathcal G)\} \subseteq \mathcal G/\Phi(\mathcal G)$ and  maximal subgroups of $G$.

  %the ordered sets $\{\bar{T} \ closed\ subgroups\ of\ {Gal}(G)/\Phi({Gal}(G))\} \cong \{\ subgroups \ of \ G\}$
  %map $H$ equivalent class of closed subgroup of $Gal(G)$ $\mapsto$ $hat{H} = \{ a \in G : \forall \sigma \in H <\sigma/\Phi(G),a> =0 \}$ so we get a bijection betteween .

\begin{prop} \label{pairing-prop}
The perfect pairing $\hat\phi_G:\mathcal G/\Phi(\mathcal G)\times G\rightarrow\mathbb Z_2$ gives an anti-isomorphism of complete lattices between the poset
$$\{R\subseteq\mathcal G/\Phi(\mathcal G):R\mbox{ is a closed subgroup of }\mathcal G/\Phi(\mathcal G)\}=$$
    $$\{\pi(T)\subseteq\mathcal G/\Phi(\mathcal G):T\mbox{ is a closed subgroup of }\mathcal G\},$$
    with $\pi:\mathcal G\rightarrow\mathcal G/\Phi(\mathcal G)$ being the canonical projection, and the poset
    $$\{\Delta\subseteq G:\Delta\mbox{ is a subgroup of }G\}.$$
    These anti-isomorphisms are given by the rules
    \begin{align*}
        \pi(T)&\mapsto\pi(T)^\perp:=\{a\in G:
        \hat\phi_G(\sigma/\Phi(G),a)=0\mbox{for all } \sigma\in T\} \\
        H&\mapsto H^\perp:=\{\sigma/\Phi(G):\hat\phi_G(\sigma/\Phi(G),a)=0\mbox{ for all } a\in H\}.
    \end{align*}
From this we get:
\begin{enumerate}[i -]
    \item  An anti-isomorphism of complete lattices between the posets
    $$\{\Delta\subseteq G:\Delta\mbox{ is a subgroup of }G\}$$
    and
    $$\{T\subseteq\mathcal G:T\mbox{ is a closed subgroup and }\Phi(\mathcal G) \subseteq T\}.$$

    \item A bijection between the sets
    $$\{\Delta\subseteq G:\Delta\mbox{ is a maximal subgroup of }G\}$$
    and
   $$\{\pi(T)\subseteq\mathcal G/\Phi(\mathcal G):T\mbox{ is a discrete subgroup with order }\ 2\} = $$
      $$\{\pi(T)\subseteq\mathcal G/\Phi(\mathcal G):T\mbox{ is a closed subgroup with }\pi(T)=\{id,\sigma/\Phi\}\mbox{, for some }\sigma\in  {\mathcal G} \setminus \Phi({\mathcal G})\}.$$
\end{enumerate}
\end{prop}

%(i) The ''perfect pairing'' $\widehat{\phi}_{G} : {\mathcal G}/{\Phi(\mathcal G)} \times G \rightarrow {\mathbb Z_2}$ gives a antiismorphism of complete lattices between the ordered sets $\bar{T} \subseteq {\mathcal G}/{\Phi(\mathcal G)} : T$ is a closed subgroup of ${\mathcal G}/{\Phi(\mathcal G)}\}$ and $\{ \Delta \subseteq G: \Delta$ is a  subgroup of $\mathcal G$  $\bar{T}$ $\mapsto$ $\bar{T}^{\perp} = \{ a \in G : \forall \sigma \in  T  <\sigma/\Phi(G),a> =0 \}$ , $H \mapsto$ $H^{\perp} = \{ \sigma/\Phi(G) : \forall a \in H <\sigma/\Phi(G)> = 0\}$.\\
%From this we get:\\
%$\ast$ \ A  antiisomorphism of complete lattices between the ordered sets: $\{\Delta \subseteq G : \Delta$ is a subgroup of  $ G\}$ and $\{T \subseteq {\mathcal G} : T$ is a closed subgroup of ${\mathcal G}$, $\Phi({\mathcal G}) \subseteq T\}$. \\
%$\ast$ \  A bijection between the sets  $\{\Delta \subseteq G : \Delta$ is a {\em maximal} subgroup of  $ G\}$ and $\{\bar{T} \subseteq {\mathcal G}/{\Phi(\mathcal G)} : \bar{T}$ is a closed subgroup of ${\mathcal G}/{\Phi(\mathcal G)}$, $\bar{T}$ has 2 elements $\}$ $\bar{T} \cong \{id, \sigma/Phi\}$
%$\ast$ \ For each $\sigma \in {\mathcal G} - \Phi({\mathcal G})$ if there is an involution in $\sigma.\Phi(\mathcal G)$ then of all element in $asifgam\Phi({\mathcal G})$ are involutions
% Moreover,  this map restricts to a bijection between discrete subgroups of order 2  $\{id, \sigma/\Phi(\mathcal G)\} \subseteq \mathcal G/\Phi(\mathcal G)$ and  maximal subgroups of $G$.
\begin{proof}

All items are immediate consequences of the isomorphism.

Since ${\mathcal G}$ is a compact Hausdorff group and $\Phi({\mathcal G}) \subseteq {\mathcal G}$ is a closed normal subgroup,  note that then the quotient map $ {\mathcal G} \twoheadrightarrow {\mathcal G}/{\Phi(\mathcal G)}$ gives an isomorphism of complete lattices between the poset of closed subgroups of ${\mathcal G}$ which contains $\Phi({\mathcal G})$  and the poset of closed subgroups of ${\mathcal G}/{\Phi(\mathcal G)}$.
\end{proof}

\begin{rem}
Let $\sigma \in {\mathcal G} \setminus \Phi({\mathcal G})$. It follows from  the definition of pairing  that:
\begin{itemize}
\item For any $x \in G-\{1\}$:  $\sigma \in M_{x}$ iff $<\sigma/\Phi({\mathcal G}),x> =0$ iff $x \in \{\Phi({\mathcal G}),\sigma.\Phi({\mathcal G})\}^\perp$.

\item If there is an involution in $\sigma.\Phi(\mathcal G)$ then of all element in $\sigma.\Phi({\mathcal G})$ are involutions.
\end{itemize}
\end{rem}

To obtain quadratic information from the Galois groups, we will need develop deeper group theoretic results.

\begin{lem} \label{Zd-le}
Let $B$ an well ordered basis of $G$ and  consider $\eta_{B} = \pi_B^{-1} : \mathcal W(B)/\mathcal V(B) \ \overset{\cong}\to \ Gal(G)$.
\begin{enumerate}[i- ]
    \item Let $a \neq 1$, choose $B= \{a_i: i\in I\}$ an well ordered basis of $G$ such that $a \in B$, say $a = a_{i}$.  Then $\eta_{B}[M'_{i}/\mathcal V(B)] = M_{a}$.

    \item Let $a,b \neq 1$ , $a\neq b$ so $\{a,b\}$ is a $\mathbb Z_2$-l.i. subset of $G$, choose $B= \{a_{i} : i\in I\}$ an well ordered basis of $G$ such that $a,b \in B$, say $a = a_{i}$, $b =a_{j}$ , $i<j \in I$. Let $\{M'_{i},M'_{j},M'\}$ be the three maximal subgroups of $\mathcal W(B)$ above $M'_{i}\cap M'_{j}$. Then $\eta_{B}[M'_{i}/{\mathcal V}(B)] = M_{a}$, $\eta_{B}[M'_{j}/\mathcal V(B)] = M_{b}$ and $\eta_{B}[M'/\mathcal V(B)] = M_{ab}$.

    \item Let $\{M_{1}, M_{2} ,M_{3}\} \subseteq Gal(G)$ maximal subgroups that are pairwise distinct.  Then are equivalent:

    $\bullet$\ $\{M_{1}, M_{2}, M_{3}\}$ are independent, which means that for {\em each} of 3 enumerations $\{u,v,w\}$ of $\{1,2,3\}$, $M_{u} \cap M_{v}  \nsubseteq M_{w}$.

    $\bullet$\  There is {\em some} enumeration $\{u,v,w\}$ of $\{1,2,3\}$ with $M_{u} \cap M_{v} \nsubseteq M_{w}$.

    $\bullet$\ $Gal(G)/(M_{1}\cap M_{2} \cap M_{3}) \cong\mathbb Z_2 \times\mathbb Z_2 \times \mathbb Z_2$
\end{enumerate}
\end{lem}

\begin{proof}
$ $
\begin{enumerate}[i- ]
\item Recall that $\mathcal V(B) \subseteq \Phi(B) \subseteq M'_{k}, \forall  k \in I$, $M'_{i} = \{\sigma \in \mathcal W(B) : \gamma_{i}(\sigma) = 0\}$. The "isomorphic" perfect pairings,
\begin{align*}
 <,>_{B} &: \mathcal W(B)/\Phi(B) \times P_{1}(B) \to {\mathbb Z}_2 \\
 <,> &: Gal(G)/\Phi(Gal(G))\times G \to \mathbb Z_2
\end{align*}
provides $M'_{i}/\Phi(B) = \{z_{i}\}^{\perp}$ and $M_{a_{i}}/\Phi(Gal(G)) = \{a_{i}\}^{\perp}$. Since the pairings are ``compatible'', i.e., the dual of the isomorphism $P_{1}(B) \overset{\cong}\to G$ : $z_{k} \mapsto a_{k}$, $k \in I$ corresponds to the isomorphism $$\mathcal W(B)/\Phi(B) \underset{can}{\overset{\cong}\to} (\mathcal W(B)/\mathcal V(B))/\Phi((\mathcal W(B)/{\mathcal V}(B))) \underset{\bar{\eta}_{B}}{\overset{\cong}\to} Gal(G)/\Phi(Gal(G)),$$
we have $\bar{\eta}_{B} \circ can [M'_{i}/\Phi(B)] = M_{a_{i}}/\Phi(Gal(G))$. Therefore, as the Frattini subgroups are contained in all maximal open subgroups, we get $\eta_{B}[M'_{i}/{\mathcal V}(B)] = M_{a_{i}}$.

\item $M' = \{\sigma \in \mathcal W(B) : \gamma_{i}(\sigma) +\gamma_{j}(\sigma) = 0\}$. The "isomorphic" perfect pairings,
$$\mathcal W(B)/\Phi(B) \times P_{1}(B) \to {\mathbb Z}_2\mbox{ and }{Gal(G)/\Phi(Gal(G))} \times G \to {\mathbb Z}_2$$
provides $M_{a}/\Phi(Gal(G)) = \{a\}^{\perp}$ and $M_{b}/\Phi(Gal(G)) = \{b\}^{\perp}$, so
\begin{align*}
    M_{ab}/\Phi(Gal(G)) &= \{ab\}^{\perp} = \{ \theta/\Phi(Gal(G)) : <\theta/\Phi(Gal(G)), ab> =0\} \\
    &\subseteq (M_{a})/\Phi(Gal(G)) \cap M_{b})/\Phi(Gal(G)))= (M_{a}\cap M_{b})/\Phi(Gal(G))
\end{align*}
 Therefore $\{M_a,M_b,M_{ab}\}$  {\em are the three maximal subgroups of $Gal(G)$ above $M_a \cap M_b$}. Since  $\{M'_{i},M'_{j},M'\}$ are the three maximal subgroups of $\mathcal W(B)$ above $M'_{i}\cap M'_{j}$ and
 $$\eta_{B}[M'_{i}/{\mathcal V}(B)] = M_{a}\mbox{ and }\eta_{B}[M'_{j}/{\mathcal V}(B)] = M_{b},$$
 we must have $\eta_{B}[M'/{\mathcal V}(B)] = M_{ab}$

\item We have uniquely determined $\{a,b,c \} \subseteq G \setminus\{1\}$, with $M_{1}= M_a , M_2= M_b , M_3 =M_c$ and, from the hypothesis, $a,b,c$ are pairwise distinct so the result follows from (ii) since if $\{x, y\}$ is  a ${\mathbb Z}_2$-li set then $\{x, xy\}$ and $\{y, xy\}$ are ${\mathbb Z}_2$-li sets and those 3 sets are ${\mathbb Z}_2$-basis of the group $\{1,x,y,xy\}$.
\end{enumerate}
\end{proof}

\begin{lem}\label{ZqDq-le}
Let $G$ be a $k$-stable pre-special group and denote $\mathcal G := Gal(G)$.
$ $
\begin{enumerate}[i -]
\item  Let $S \subseteq\mathcal G$ a normal closed  of with  $\mathcal G/S \cong \mathbb Z_4$ then there is a unique maximal subgroup  $H \subseteq\mathcal G$ such that $S \subseteq H$.

\item Let $D \subseteq\mathcal G$ a normal closed  of with  $\mathcal G/S \cong\mathbb D_4$ then there is a unique set $\{H_{1}, H_{2}\}$ with $H_{1} \neq H_{2}$, $H_{i} \subseteq\mathcal G$  maximal subgroups  such that $D \subseteq H_{1} \cap H_{2}$ and if $\{H, H_{1}, H_{2}\}$ are the {\em three} maximal subgroups above $H_{1}\cap H_{2}$ then $H/D \cong\mathbb Z_4$.
\end{enumerate}
\end{lem}

\begin{proof}
$ $
\begin{enumerate}[i -]
    \item For the existential  part take $r \in\mathcal G\setminus S$ such that $r^{2} \notin S$ and  $\mathcal G/S = \{1.S,r.S,r^{2}.S,r^{3}.S\}$ $ \cong \mathbb Z_4$ and take the maximal subgroup $H = 1S \cup r^{2}S$. Then $ S \subseteq H$ and the canonical epimorphism $\mathcal G/S \twoheadrightarrow\mathcal G/H$ corresponds to the epimorphism $\mathbb Z_4 \twoheadrightarrow\mathbb Z_2$. This maximal $H \supseteq S$ is unique because if $S \subseteq H_{1} , H_{2}$ with $H_{1} \neq H_{2}$, then $S \subseteq H_{1} \cap H_{2}$ and there will be an epimorphism $\mathbb Z_4 \cong\mathcal G/S \twoheadrightarrow\mathcal G/H_{1}\cap H_{2} \cong\mathbb Z_2\times \mathbb Z_2$, however there is no element in $\mathbb Z_2 \times\mathbb Z_2$ of order 4.

    \item For the existential  part let $r,s \in\mathcal G\setminus D$ be such that $r^{2} \notin D, s^{2} \in D$ and then
    \begin{align*}
     \mathbb D_4 &\cong \mathcal G/D=\{1.D,r.D,r^{2}.D,r^{3}.D,s.D,sr.D,sr^{2}.D,sr^{3}.D\}
    \end{align*}
    Then each one of the maximal subgroups above $D$ is a reunion of four classes so they must contain $1D, r^{2}D$. they are three:
    \begin{align*}
        1D \cup r^{2}D \cup rD \cup r^{3}D \\
        1D \cup r^{2}D \cup sD\cup sr^{2}D \\
        1D \cup r^{2}D \cup srD \cup sr^{3}D
    \end{align*}
    Then take
    \begin{align*}
     \{H_{1}, H_{2} \} &= \{1D \cup r^{2}D \cup sD\cup sr^{2}D, 1D \cup r^{2}D \cup srD \cup sr^{3}D\} \mbox{ and }\\
     H &= 1D \cup r^{2}D \cup rD \cup r^{3}D.
    \end{align*}
    Then $ D \subseteq H_1\cap H_2 = 1D \cup r^{2}D \subseteq H, H_{1}, H_{2}$. Then the canonical epimorphism $\mathcal G/D \twoheadrightarrow\mathcal G/H_{1} \cap H_{2}$ corresponds to the epimorphism $\mathbb D_4 \twoheadrightarrow\mathbb Z_2 \times\mathbb Z_2$ and $H/D \cong\mathbb Z_4$ . This pair of maximals $H_1, H_{2} \supseteq D$ is unique because if $H_{3}$ is a maximal $S \subseteq H_{3}$ with $H_{3} \neq H_{1}, H_{2}$, then $D \subseteq H_{1} \cap H_{2}\cap H_{3}$ and we have two  to consider:\\
$\bullet$ \ $H_{1} \cap H_{2} \subseteq H_{3}$:  in this case $H_{3} = H = 1D \cup r^{2}D \cup rD \cup r^{3}D$, then
$$1D =\cup r^{2}D = H_{1}\cap H_{2} = H\cap H_{1} = H\cap H_{2}$$
so $D \subseteq H\cap H_{1}$, but we see directly that $H_{2}/D \ncong{\mathbb Z}_4$; similarly $D \subseteq H\cap H_{2}$ but also $H_{1}/D \ncong{\mathbb Z}_4$;\\
$\bullet$ \ $H_{1} \cap H_{2} \nsubseteq H_{3}$ then $H_{3} \neq H,H_{1},H_{2}$, also $H_{1} \cap H_{3} \nsubseteq H_{2}$ (because if $H_{1} \cap H_{3} \subseteq H_{2}$, then $H_{1} \cap H_{2} \nsubseteq H_{3}$, see the previous Lemma) and similarly $H_{2} \cap H_{3} \nsubseteq H_{1}$, then $\{H_{1},H_{2},H_{3}\}$ are {\em independent}, so in this case, the epimorphism $\mathcal G/D \twoheadrightarrow \mathcal G/H_{1}\cap H_{2}\cap H_{3}$ corresponds to an epimorphism ${\mathbb D}_4 \twoheadrightarrow {\mathbb Z}_2\times {\mathbb Z}_2 \times {\mathbb Z}_2$ but there is no element in ${\mathbb Z}_2 \times {\mathbb Z}_2 \times {\mathbb Z}_2$ of order 4.
\end{enumerate}
\end{proof}

\begin{teo}\label{standardteo}
$ $
\begin{enumerate}[i -]
    \item %There is an injective map
    %$$\{a\in G\setminus\{1\}:l(a)l(a)=0\in k_2(G)\}\rightarrow\mbox{Im}(j_1).$$
    Let $a \in G \setminus\{1\}$.
    $$l(a).l(a) = 0 \in k_2(G) \ \Rightarrow$$
    There is $S \subseteq Gal(G)$, a  normal clopen subgroup such that $Gal(G)/S \cong {\mathbb Z}_4$ and $S \subseteq M_{a}$.
    \item %There is an injective map
    %$$\{a,b\in G:\{a,b\}\mbox{ is linearly independent and }l(a)l(b)=0\in k_2(G)\}\rightarrow\mbox{Im}(j_2).$$
     Let $a, b \in G \setminus \{1\}$ such that $a \neq b$
    $$l(a).l(b) = 0 \in k_2(G) \ \Rightarrow$$
    There is $D \subseteq Gal(G)$, a normal clopen subgroup such that $Gal(G)/D \cong {\mathbb D}_4$ and $D \subseteq M_{a} \cap M_{b}$, $M_{ab}/D \cong {\mathbb Z}_4$.
\end{enumerate}
\end{teo}
\begin{proof}
$ $
\begin{enumerate}[i -]
    \item Since $\{a\} \subseteq G$  is a l.i. subset, take an  well ordered basis $B = \{a_{i} : i \in I\} \subseteq G$ such that $a \in B$, say $a = a_{i}$. Then $\eta_{B} : {\mathcal W}(B)/{\mathcal V}(B) \overset{\cong}\to \ Gal(G)$ with ${\mathcal V}(B) \subseteq \Phi(B)$ and ${\mathcal V}(B) = Q(B)^{\perp}$ where $Q(B) = ker(P_{2}(B) \twoheadrightarrow k_{2}(G))$, then, by Proposition \ref{fixsg1}, $Q(B) = [\{q^{B}_{xy} : l(x)l(y) = 0\}]$ so
    $${\mathcal V}(B) = \bigcap\{(q^{B}_{xy})^{\perp} : l(x)l(y) =0\} = \bigcap\{(q^{B}_{xy})^{\perp} : x, y \neq 1 , l(x)l(y) =0\}.$$ Denote $M'_{i} = \{ \sigma \in {\mathcal W}(B) : \gamma_{i}(\sigma) =0\}$ and $S'_{i} = \{ \sigma \in {\mathcal W}(B) : \alpha_{i}(\sigma) =\gamma_{i}(\sigma) = 0\}$ then, by Proposition \ref{fixms2}(i), $S'_{i} \subseteq M'_{i} \subseteq {\mathcal W}(B)$ are clopen normal subgroups with ${\mathcal W}(B)/M'_{i}  \cong {\mathbb Z}_2$, ${\mathcal W}(B)/S'_{i} \cong {\mathbb Z}_4$. We have ${\mathcal V}(B) \subseteq \Phi(B) \subseteq M'_{i}$, and we state the

$ $

{\underline{Claim}:}  ${\mathcal V}(B) \subseteq S'_{i}$.

This entails that $M_{a} = \eta_{B}[M'_{i}/{\mathcal V}(B)] \subseteq Gal(G)$, $Gal(G)/M_{a} \cong {\mathbb Z}_2$ and
$$S := \eta_{B}[S'_{i}/{\mathcal V}(B)] \subseteq Gal(G)$$
is a clopen normal subgroup of $Gal(G)$ with $Gal(G)/S \cong {\mathbb Z}_4$ and $S \subseteq M_{a}$, as we need.

$ $

{\underline{Proof of the Claim}}: We will see that $S'_{i} \cap \Phi(B) = (q^{B}_{a_{i}a_{i}})^{\perp}$ then as $a = a_{i}$ and $l(a)l(a) =0$ we get ${\mathcal V}(B) \subseteq (q^{B}_{a_{i}a_{i}})^{\perp}$ so ${\mathcal V}(B) \subseteq S'_{i} \cap \Phi(B)$. Since $1 \neq a =a_{i}$ it follows that $q^{B}_{a_{i} a_{i}} = z^{2}_{i} \in P_{2}(B)$ is such that $(q^{B}_{a_{i}a_{i}})^{\perp} \subseteq \Phi(B)$ has {\em index 2} and we will proof that $S'_{i} \cap \Phi(B) \subseteq \Phi(B)$ has also {\em index 2} and $S'_{i} \cap \Phi(B) \subseteq (q^{B}_{a_{i}a_{i}})^{\perp}$ so  we get  $S'_{i} \cap \Phi(B) = (q^{B}_{a_{i}a_{i}})^{\perp}$. Firstly we show  that $\Phi(B)/S'_{i} \cap \Phi(B) \cong{\mathbb Z}_2$: as $\Phi(B) \hookrightarrow M'_{i}$ then $\Phi(B)/S'_{i} \cap \Phi(B) \rightarrowtail M'_{i}/S'_{i}$ and $M'_{i}/S'_{i} \cong {\mathbb Z}_2$ so $\Phi(B)/S'_{i} \cap \Phi(B)$ has 1 or 2 elements.  However, it cannot has 1 element: if $S'_{i} \cap \Phi(B) = \Phi(B)$ then, by Proposition \ref{fixms5}, $$\bigcap\{ M'_{j} : j \in I\} = \Phi(B) \subseteq S'_{i}$$
but ${\mathcal W}(B)$ is a compact space and $S'_{i} \subseteq {\mathcal W}(B)$ is open subset , $M'_{j} \subseteq {\mathcal W}(B)$ is a closed subset $j \in I$ so there is a {\em finite subset} $\{j_{0},\ldots, j_{n}\} \subseteq I$ such that $M'_{j_{0}} \cap \ldots \cap M'_{j_{n}} \subseteq S'_{i}$, choose $n \in {\mathbb N} $ {\em minimum} with this property so for each $m \leq n$ , $\bigcap\{M'_{j_{l}} : l \neq m\} \nsubseteq M'_{j_{m}}$ then we have an {\em isomorphism} $$Gal(G)/\bigcap\{M'_{j_{l}} : l\leq n\} \overset{\cong}\to \prod_{l \leq n} Gal(G)/M'_{j_{l}}$$
so the epimorphism $Gal(G)/\bigcap\{M'_{j_{l}} : l\leq n\} \twoheadrightarrow Gal(G)/S'_{i}$ corresponds to an epimorphism $\prod_{l \leq n} {\mathbb Z}_2 \twoheadrightarrow {\mathbb Z}_4$, but the two elements of order 4 in ${\mathbb Z}_4$ cannot be in the image of the homomorphism. Now we prove that $S'_{i} \cap \Phi(B) \subseteq (q^{B}_{a_{i}a_{i}})^{\perp}$: we have $(q^{B}_{a_{i}a_{i}})^{\perp}  = \{z^{2}_{i}\}^{\perp}$ and $$S'_{i} \cap \Phi(B) = \{ \sigma \in {\mathcal W}(I) : \alpha_{i}(\sigma) =0\mbox{ and }\gamma_{j}(\sigma) = 0\mbox{ for each }j \in I\}$$
and it follows from of the group operation and the definition of the pairing $<,> : \Phi(B) \times P_{2}(B) \to {\mathbb Z}_2$ that
$$\{x_{k}x_{l} : k <l \in I\} \cup\{x^{2}_{j} : i \neq j \in I\}\subseteq \left(S'_{i} \cap \Phi(B)\right)\cap \{z^{2}_{i}\}^{\perp}$$
 Since $S'_{i} \cap\Phi(B) ,\{z^{2}_{i}\}^{\perp}  \subseteq \Phi(B)$ are closed subgroups,
 $$closure([\{x_{k}x_{l} : k <l \in I\} \cup\{x^{2}_{j} : i \neq j \in I\}])\subseteq \left(S'_{i} \cap \Phi(B)\right)\cap \{z^{2}_{i}\}^{\perp}.$$
 Now we will  prove that $S'_{i} \cap \Phi(B) \subseteq closure([\{x_{k}x_{l} : k <l \in I\} \cup\{x^{2}_{j} : i \neq j \in I\}])$; it is enough find for each $\sigma  \in S'_{i} \cap \Phi(B)$ and each {\em basic} neighborhood $T$ of $1 \in {\mathcal W}(B)$ two finite sets $\{j_{1}, \ldots j_{n}\} \subseteq I-\{i\}$ and $\{ (k_1,l_{1}) , \ldots (k_{m},l_{m}) : k_{u} < l_{u} \in I , 1\leq u \leq m\}$  such that $(x^{2}_{j_{1}}. \ldots .x^{2}_{j_{n}}. x_{k_{1}}x_{l_{1}}. \ldots . x_{k_{m}}.x_{l_{m}}) \in \sigma. T$ : let $T = \bigcap U$ where
 \begin{align*}
     U \subseteq_{fin} V&=
     \{M'_j: j \in I\}\cup\{S'_j: j \in I\}\cup\{D'_{kl}:k<l\in I\}
 \end{align*}
 and take
 \begin{align*}
  \{j_{1}, \ldots j_{n}\} &= \{ j \in I : S'_{j} \in U , \alpha_{j}(\sigma) =1\} \subseteq I \setminus\{i\}\mbox{ and }\\
  \{ k_{1} < l_{1}, \ldots k_{m} < l_{m} \} &= \{ k < l \in I: D'_{kl} \in U, \beta_{kl}(\sigma) = 1\}
 \end{align*}
 $\{j_{1}, \ldots j_{n}\} = \{ j \in I : S'_{j} \in U , \alpha_{j}(\sigma) =1\} \subseteq I \setminus\{i\}$ and $\{ k_{1} < l_{1}, \ldots k_{m} < l_{m} \} = \{ k < l \in I: D'_{kl} \in U, \beta_{kl}(\sigma) = 1\}$ then,  since $\gamma_{j}(\sigma) = 0$, for all $j \in I$, we get
$$(x^{2}_{j_{1}}. \ldots .x^{2}_{j_{n}}. x_{k_{1}}x_{l_{1}}. \ldots . x_{k_{m}}.x_{l_{m}}) \in \sigma. T$$

    \item  Since $\{a,b\} \subseteq G$  is a  l.i. subset, take an well ordered basis $B = \{a_{i} : i \in I\} \subseteq G$ such that $a,b \in B$, say $a = a_{i} , b = a_{j} , i<j \in I$. %Then $\eta_{B} : {\mathcal W}(B)/{\mathcal V}(B) \overset{\cong}\to \ Gal(G)$ with ${\mathcal V}(B) \subseteq \Phi(B)$ and ${\mathcal V}(B) = Q(B)^{\perp}$ where $Q(B) = ker(P_{2}(B) \twoheadrightarrow k_{2}(G))$, by Lemma (\kmar{colocar ref}) $Q(B) = [\{q^{B}_{xy} : l(x)l(y) = 0\}]$ so ${\mathcal V}(B) = \bigcap\{(q^{B}_{xy})^{\perp} : l(x)l(y) =0\} = \bigcap\{(q^{B}_{xy})^{\perp} : x, y \neq 1 , l(x)l(y) =0\}$ ; let us
    We denote
    \begin{align*}
        M'_{i} &= \{ \sigma \in {\mathcal W}(B) : \gamma_{i}(\sigma) =0\}\\
        M'_{j} &= \{ \sigma \in {\mathcal W}(B) : \gamma_{j}(\sigma) =0\}\\
        M' &= \{\sigma \in {\mathcal W}(B) : \gamma_{i}(\sigma) + \gamma_{j}(\sigma) = 0 \}\\
        D'_{ij} &= \{ \sigma \in {\mathcal W}(B) : \beta_{ij}(\sigma) =\gamma_{i}(\sigma) =\gamma_{j}(\sigma) = 0\}
    \end{align*}
    By Proposition \ref{fixms2}(ii), $D'_{ij} \subseteq M'_{i}, M'_{j} \subseteq {\mathcal W}(B)$ are clopen normal subgroups with ${\mathcal W}(B)/M'_{i} \cap M'_{j}  \cong {\mathbb Z}_2 \times {\mathbb Z}_2$, ${\mathcal W}(B)/D'_{ij} \cong {\mathbb D}_4$. Besides ${\mathcal V}(B) \subseteq \Phi(B) \subseteq M'_{i} \cap M'_{j}$ %, by Lemma (\kmar{colocar ref}),
and we stat the

$ $

{\underline{Claim}}: ${\mathcal V}(B) \subseteq D'_{ij}$.

This entails that $M_{a} = \eta_{B}[M'_{i}/{\mathcal V}(B)] \subseteq Gal(G)$, $M_{b} = \eta_{B}[M'_{j}/{\mathcal V}(B)]$, $M_{ab} = \eta_{B}[M'/{\mathcal V}(B)]$ , $Gal(G)/M_{a} \cong {\mathbb Z}_2 \cong Gal(G)/M_{b}$ and $D \doteq \eta_{B}[S'_{i}/{\mathcal V}(B)] \subseteq Gal(G)$ is a clopen normal subgroup of $Gal(G)$ with $Gal(G)/D \cong {\mathbb D}_4$ such that $D \subseteq M_{a} \cap M_{b}$ and $M_{ab}/D \cong {\mathbb Z}_4$, as we need.

$ $

{\underline{Proof of the Claim}}: We will see that $D'_{ij} \cap \Phi(B) = (q^{B}_{a_{i}a_{j}})^{\perp}$ then as $a = a_{i} , b =a_{j}$ and $l(a)l(b) =0$ we get ${\mathcal V}(B) \subseteq (q^{B}_{a_{i}a_{j}})^{\perp}$ so ${\mathcal V}(B) \subseteq D'_{ij} \cap \Phi(B)$. As $1 \neq a =a_{i}$ and $1 \neq b = a_{j}$ with $i<j \in I$, it follows that $q^{B}_{a_{i} a_{j}} = z_{i}z_{j} \in P_{2}(B)$ is such that $(q^{B}_{a_{i}a_{j}})^{\perp} \subseteq \Phi(B)$ has {\em index 2} and we will proof that $D'_{ij} \cap \Phi(B) \subseteq \Phi(B)$ has also {\em index 2} and $D'_{ij} \cap \Phi(B) \subseteq (q^{B}_{a_{i}a_{j}})^{\perp}$ so  we get  $D'_{ij} \cap \Phi(B) = (q^{B}_{a_{i}a_{j}})^{\perp}$.

Firstly, we will prove that $\Phi(B)/D'_{ij} \cap \Phi(B) \cong{\mathbb Z}_2$: since $\Phi(B) \hookrightarrow M'_{i} \cap M'_{j}$ then
$$\Phi(B)/D'_{ij} \cap \Phi(B) \rightarrowtail M'_{i}\cap M'_{j}/D'_{ij}\mbox{ and }M'_{i}\cap M'_{j}/D'_{ij} \cong {\mathbb Z}_2$$
so $\Phi(B)/D'_{ij} \cap \Phi(B)$ has 1 or 2 elements. However it cannot has 1 element: if $D'_{ij} \cap \Phi(B) = \Phi(B)$ then  %, by Lemma (\kmar{colocar ref}).
$\bigcap\{ M'_{k} : k \in I\} = \Phi(B) \subseteq D'_{ij}$ but ${\mathcal W}(B)$ is a compact space and $D'_{ij} \subseteq {\mathcal W}(B)$ is open subset , $M'_{k} \subseteq {\mathcal W}(B)$ is a closed subset $k \in I$ so there is a {\em finite subset} $\{k_{0},\ldots, k_{n}\} \subseteq I$ such that $M'_{k_{0}} \cap \ldots \cap M'_{k_{n}} \subseteq D'_{ij}$, choose $n \in {\mathbb N} $ {\em minimum} with this property so for each $m \leq n$ , $\bigcap\{M'_{k_{l}} : l \neq m\} \nsubseteq M'_{k_{m}}$ then we have an {\em isomorphism} $$Gal(G)/\bigcap\{M'_{k_{l}} : l\leq n\} \overset{\cong}\to \prod_{l \leq n} Gal(G)/M'_{k_{l}}$$
so the epimorphism $Gal(G)/\bigcap\{M'_{k_{l}} : l\leq n\} \twoheadrightarrow Gal(G)/D'_{ij}$ corresponds to an epimorphism $\prod_{l \leq n} {\mathbb Z}_2 \twoheadrightarrow {\mathbb D}_4$, but the two elements of order 4 in ${\mathbb D}_4$ cannot be in the image of the homomorphism.

Now we prove that $D'_{ij} \cap \Phi(B) \subseteq (q^{B}_{a_{i}a_{j}})^{\perp}$: we have $(q^{B}_{a_{i}a_{j}})^{\perp}  = \{z_{i}z_{j}\}^{\perp}$ and
$$D'_{ij} \cap \Phi(B) = \{ \sigma \in {\mathcal W}(I) : \beta_{ij}(\sigma) =0\mbox{ and }\gamma_{k}(\sigma) = 0\mbox{ for each }k \in I\}$$
 and it follows from of the group operation and the definition of the pairing
 $$<,> : \Phi(B) \times P_{2}(B) \to {\mathbb Z}_2$$
 that
$$\{x_{k}x_{l} : k <l \in I , (k,l) \neq (i,j)\} \cup\{x^{2}_{k} : k \in I\} \subseteq \left(D'_{ij} \cap \Phi(B)\right)\cap\{z_{i}z_{j}\}^{\perp}$$
then, since $D'_{ij} \cap\Phi(B) ,\{z_{i}z_{j}\}^{\perp}  \subseteq \Phi(B)$ are closed subgroups,
$$closure([\{x_{k}x_{l} : k <l \in I , (k,l) \neq (i,j)\} \cup\{x^{2}_{k} : k \in I\}]) \subseteq\left(D'_{ij}\cap\Phi(B)\right)\cap\{z_{i}z_{j}\}^{\perp}.$$
Now we will  prove that $D'_{ij} \cap \Phi(B) \subseteq closure([\{x_{k}x_{l} : k <l \in I, (k,l) \neq (i,j)\} \cup\{x^{2}_{k} : k \in I\}])$; it is enough find for each $\sigma  \in D'_{ij} \cap \Phi(B)$ and each {\em basic} neighborhood $T$ of $1 \in {\mathcal W}(B)$ two finite sets $\{j_{1}, \ldots, j_{n}\} \subseteq I$ and $\{ (k_1,l_{1}) , \ldots (k_{m},l_{m}) : k_{u} < l_{u} \in I , (k_{u},l_{u}) \neq (i,j), 1\leq u \leq m \}$  such that $(x^{2}_{j_{1}}. \ldots .x^{2}_{j_{n}}. x_{k_{1}}x_{l_{1}}. \ldots . x_{k_{m}}.x_{l_{m}}) \in \sigma. T$ : let $T = \bigcap U$ where
\begin{align*}
    U \subseteq_{fin} V&=
    \{M'_j: j \in I\}\cup\{S'_j: j \in I\}\cup\{D'_{kl}:k<l\in I\}
\end{align*}
and take
\begin{align*}
 \{j_{1}, \ldots j_{n}\} &= \{ j \in I : S'_{j} \in U , \alpha_{j}(\sigma) =1\} \subseteq I\mbox{ and }\\
 \{ k_{1} < l_{1}, \ldots k_{m} < l_{m} \} &= \{ k < l \in I: D'_{kl} \in U, \beta_{kl}(\sigma) = 1\} \subseteq I\times I \setminus\{(i,j)\}
\end{align*}
 then, since $\gamma_{k}(\sigma) = 0$, for all  $k \in I$, we get
 $$(x^{2}_{j_{1}}. \ldots .x^{2}_{j_{n}}. x_{k_{1}}x_{l_{1}}. \ldots . x_{k_{m}}.x_{l_{m}}) \in \sigma. T.$$
\end{enumerate}
\end{proof}

The above proposition suggests the following:

\begin{defn} \label{pSGstandard-def}
A  pre-special group $G$ is said to be \textbf{standard} if it is a $k$-stable pre-special group and holds both the reverse implications in the  Theorem \ref{standardteo} above.
\end{defn}

\begin{rem}
Lemma \ref{ZqDq-le} determines (injective) maps
\begin{align*}
    j_1&:\{S\subseteq\mathcal G:S\mbox{ is a normal subgroup of index }\mathbb Z_4\}\rightarrow \\
    &\{M\subseteq\mathcal G:M\mbox{ is a maximal subgroup}\}; \\
    j_2&:\{D\subseteq\mathcal G:D\mbox{ is a normal subgroup of index }\mathbb{D}_4\}\rightarrow \\
    &\qquad\{\{M_1,M_2\}:M_1,M_2\subseteq\mathcal G,\,M_1\ne M_2\mbox{ are maximal subgroups}\}.
\end{align*}

By the canonical bijection $\mathbb{M} : G\setminus \{1\} \ \overset{\cong}\to \ \{M \subseteq Gal(G): M$ is a maximal clopen subgroup $\}$, it is natural to ask: \\
(1) Which subset of $\{ a \in G : a\neq 1\}$ corresponds bijectively with $image(j_{1})$?\\
(2) Which subset of $\{ \{a,b\} \subseteq G : a,b\neq 1 , a\neq b\}$ corresponds bijectively with $image(j_{2})$?\\
%As we will see (\cite{}, \cite{}), the answers to those questions encodes much of the structure of Galois groups of pre-special groups!

The concept of {\em standard } pre-special group provides a full answer to these questions:\\
(1) The set $\{ \{a\} \subseteq G: \{a\} l.i., l(a)l(a) = 0 \in k_{2}(G)\}$ corresponds bijectively with $image(j_{1})$.\\
(2) The set $\{ \{a,b\} \subseteq G : \{a,b\} l.i., l(a)l(b) =0 \in k_{2}(G)\}$ corresponds bijectively with $image(j_{2})$.
\end{rem}

It follows from  Propositions 2.3 and 2.4 in \cite{minac1996witt} that $SG(F)$ is a standard special group, for every field $F$ with $char(F) \neq 2$.

We have already established that every special group $G$ is $k$-stable (see Proposition \ref{k-stable-aph}(iii)).

These suggest the following:

\begin{op}
    Is every special group $G$ standard?\footnote{We are unable to solve this question with the methods so far developed. We believe that to address this question, we will have to develop the theory of quadratic extensions of pre-special hyperfields.}
\end{op}

In the sequel, we will see the relevance of the subclass of standard pre-special groups. We invite the reader to recall Proposition \ref{pairing-prop}.

\begin{prop}\label{encoding-pr}
Let $G$ be a $k$-stable pre special group and denote ${\mathcal G} := Gal(G)$.
\begin{enumerate}[i- ]
    \item Let  $\sigma \in {\mathcal G} \setminus \Phi({\mathcal G})$ be such that $\sigma^{2} =id$. Then $\{\Phi({\mathcal G}),\sigma.\Phi({\mathcal G})\}^{\perp} \subseteq G$ is a maximal saturated subgroup of $G$. %ESTE ITEM (i) NAO PRECISA DESTA HIPOTESE K-stable
% If $G$ $k$-stable {\em pre-special} group then
\item Suppose that $G$ is a {\em standard pre-special group}. Let $\sigma \in {\mathcal G} \setminus \Phi(\mathcal G)$ be such that $\sigma^{2} \neq id$ (so $\sigma^{4} =id$). Then $\{\Phi({\mathcal G}),\sigma.\Phi({\mathcal G})\}^{\perp} \subseteq G$  is {\em not} a saturated subgroup of $G$.

%If $G$ $k$-stable {\em pre-special} group then

\item Suppose that $G$ is a {\em standard special group}. The set of all classes of conjugacy of involutions $\sigma \in {\mathcal G} \setminus \Phi(G)$ corresponds to the set of all  orderings (= maximal saturated subgroups) of $G$.
% ($H_{\sigma} = \{id, \sigma\}, \sigma^{2} =id$) (both Closed sets)

\item Suppose that $G$ is a {\em standard pre-special group}. Then we have an anti-isomorphism of complete lattices between the  posets $$\{ \Delta \subseteq G : \Delta \mbox{ is a saturated subgroups of}\ G\}$$
and
$$\{T \subseteq {\mathcal G} : T \mbox{ is a closed subgroup of}\ {\mathcal G} \mbox{(topologically) generated by involutions such that}\ \Phi({\mathcal G}) \subseteq T\}$$
\end{enumerate}
\end{prop}

\begin{proof}
$ $
\begin{enumerate}[i- ]
\item Let $\bar{T} = \{\Phi({\mathcal G}), \sigma\Phi({\mathcal G})\}$.

$ $

\underline{Claim:} To have that $\bar{T}^\perp \subseteq G$ is a {\em  saturated} subgroup it is enough to prove the following:  $\forall x, y \in G$ if $<x,y> \equiv <1,xy>$, then $x \in \bar{T}^\perp$ or $ y \in \bar{T}^\perp$ .

$ $

\underline{Proof of Claim:}  Firstly we prove that $-1 \notin \bar{T}^\perp$: take any $x \notin \bar{T}^{\perp}$ (there is some $x$,  as $\bar{T}^{\perp} \subseteq G$ has index 2) then as $<x,-x> \equiv <1,-1>$ it follows from assumption in the claim that $-x \in \bar{T}^\perp$ so if $-1 \in \bar{T}^\perp$ then $x = -1.(-x) \in \bar{T}^\perp$, a contradiction. Now let us prove that $\bar{T}^\perp$ is saturated: take any $a, b \in G$ such that $b \in D_{G}(<1,a>)$, assume $a \in \bar{T}^{\perp}$ then we have to prove that $b \in \bar{T}^\perp$: as $(-a). a = -1 \notin \bar{T}^\perp$ then $-a \notin \bar{T}^{\perp}$ and as $<b,ba>\equiv<1,a>$ we have $<b,-a> \equiv <1,-ba>$ so, by the assumption in the claim, we get $b \in {\bar{T}}^\perp$. \\
Now we will prove that $\sigma^2 = id$ entails
$\forall x, y \in G$ if $<x,y> \equiv <1,xy>$ then $x \in \bar{T}^\perp$ or $ y \in \bar{T}^\perp$:\\
We have three cases:\\
$\ast$ $x$ (or $y$) is $1$;\\
$\ast$ $x=y\neq 1$; \\
$\ast$ $x,y \neq 1$ and $x\neq y$.\\
There is nothing to proof in the first case. Now consider $x \in G \setminus\{1\}$ such that $<x,x>\equiv <1,1>$: we must prove that $x\in {\bar{T}}^\perp$.
%$G \rightarrow G^{(k)}$ is a $SG$-morphism
Since $G$ is $k$-stable, we have $l(x)l(x)=l(1)l(1) = 0$ then, by Theorem \ref{standardteo}(i), there is a $S \subseteq {\mathcal G}$ a clopen normal subgroup such that ${\mathcal G}/S \cong {\mathbb Z_4}$ and $S \subseteq M_{x}$. Consider the quotient map $p_{S}: {\mathcal G} \twoheadrightarrow {\mathcal G}/S$ and write ${\mathcal G}/S =\{1/S,r/S,r^2/S,r^3/S\}$ then as $\sigma^2 = id$ we must have $\sigma/S  \in \{1/S,r^2/S\}$. If $\sigma/S = 1/S$ then $\sigma \in S \subseteq M_{x}$ i.e. $<\sigma/\Phi({\mathcal G}),x>=0$ so $x \in \{\Phi({\mathcal G}),\sigma\Phi({\mathcal G})\}^\perp$. If $\sigma/S = r^2/S$ then $\sigma.r^2 = \sigma.r^{-2} \in S \subseteq M_{x}$ i.e. $<(\sigma.r^2)/\Phi({\mathcal G}),x>=0$ but
$$<(\sigma.r^2)/\Phi({\mathcal G}), x> = <\sigma/\Phi({\mathcal G}),x> + <r/\Phi({\mathcal G}),x> + <r/\Phi({\mathcal G}),x> = <\sigma/\Phi({\mathcal G}),x>$$
then $<\sigma/\Phi({\mathcal G}),x> =0$ so $x \in \{\Phi({\mathcal G}),\sigma\Phi({\mathcal G})\}^\perp$. Now take $x,y \in G\setminus\{1\}$ with $x\neq y$ and  $<x,y> \equiv <1,xy>$ and supose $x \notin \bar{T}^\perp$ then we must prove that $y \in \bar{T}^\perp$. As $\{x,y\}$ is a two element l.i. set and $l(x)l(y) = l(1)l(xy) = 0$ we have, by Theorem \ref{standardteo}(ii), some $D\subseteq {\mathcal G}$ a clopen normal subgroup such that ${\mathcal G}/D \cong {\mathbb D}_4$, $D \subseteq M_{x}\cap M_{y}$ and $M_{x.y}/D \cong {\mathbb Z_4}$. Consider the quotient morphism $p_{D} : {\mathcal G} \twoheadrightarrow {\mathcal G}/D$ and write
$${\mathcal G}/D = \{1/D,r/D,r^2/D,r^3/D,s/D,sr/D,sr^2/D,sr^3/D\}$$
then, as $\sigma^2 =id$, we have $\sigma/D \notin  \{r/D,r^3/D\}$. Let us prove that $\sigma/D \notin \{1/D,r^2/D\}$: as $<1/\Phi({\mathcal G}),x> = <r^2/\Phi({\mathcal G}),x> = 0$ we have $\{id, r^2 \} \subseteq M_{x}$ and as we selected $x \notin \{\Phi,\sigma\Phi\}^\perp$ we have  $\sigma \notin M_{x}$ then if $\sigma/D = r^2/D$ then $\sigma.r^{-2} \in D \subseteq M_{x}$ so
$$\sigma = (\sigma.r^{-2}).r^2 \in D.M_{x} \subseteq M_{x}.M_{x} \subseteq M_{x},$$
a contradiction; similarly $\sigma/D \neq 1/D$. So we have $\sigma/D \in \{s/D,sr/D,sr^2/D,sr^3/D\}$. Now, as $M_{x.y}/D \cong {\mathbb Z_4}$ we have   $M_{x.y} = 1D \cup r^{2}D \cup rD \cup r^{3}D $ (see the proof of Theorem \ref{standardteo}(ii)) and
$$\{M_{x}, M_{y} \} = \{1D \cup r^{2}D \cup sD\cup sr^{2}D, 1D \cup r^{2}D \cup srD \cup sr^{3}D\}.$$
If $M_{x} = 1D \cup r^{2}D \cup sD\cup sr^{2}D$ then as $\sigma \notin M_{x}$ we have $\sigma/D \notin \{s/D,sr^2/D\}$ so $$\sigma/D \in \{sr/D,sr^3/D\}\subseteq M_{y}/D$$
then $\sigma \in M_{y}$ that is $y \in \{\Phi({\mathcal G}),\sigma\Phi({\mathcal G})\}^\perp$; similarly if $M_{x} = 1D \cup r^{2}D \cup srD \cup sr^{3}D$ then $y \in \{\Phi({\mathcal G}),\sigma\Phi({\mathcal G})\}^\perp$.

\item Let $\bar{T} = \{\Phi({\mathcal G}), \sigma\Phi({\mathcal G})\}$.

$ $

{\underline{Claim}}: To have that $\bar{T}^\perp \subseteq G$ {\em is  not} a  saturated subgroup it is enough to prove the following: $\exists x, c \in G \setminus \bar{T}^\perp$ such that $<x,c> \equiv <1,xc>$.

 $ $

\underline{Proof of Claim:} If $-1 \in \bar{T}^\perp$ then $\bar{T}^\perp \subsetneq G$ so $G = D_{G}<1,-1> \nsubseteq \bar{T}^{\perp}$ so  $\bar{T}^\perp$ is not a saturated subgroup. If $-1 \notin \bar{T}^\perp$ then take $x, c \in G\setminus \bar{T}^\perp$ such that $<x,c> \equiv <1,xc>$ so we have $<c,-xc> \equiv <1,-x>$, that is $c \in D_{G}<1,-x>$ and $-x \in \bar{T}^\perp$: if $-x \notin \bar{T}^\perp$ then as $\bar{T}^\perp \subseteq G$ has index 2 $-1.\bar{T}^\perp = -x.\bar{T}^\perp$ so $x = -1.-x \in \bar{T}^\perp$;  that is we established that there are $a (= -x) \in \bar{T}^{\perp}$ and $c \in D_{G} <1,a>$ with $c \notin \bar{T}^{\perp}$: this means that $\bar{T}^\perp$ is not saturated. \\
Now we will prove that $\sigma^2 \neq id$ entails $\exists x, c \in G \setminus\bar{T}^\perp$ such that $<x,c> \equiv <1,xc>$:\\
Take $B$ any well ordered base of $G$ and consider the composition
$$\mathcal W(B) \twoheadrightarrow\mathcal W(B)/\mathcal V(B){\overset{\cong}\rightarrow}{\mathcal G}$$
and, by Lemma \ref{lifting-le}, {\em choose} any lifting $\widetilde{\sigma} \in\mathcal W(B)$ of $\sigma \in{\mathcal G}$. Since $\sigma^{2} \neq id \in {\mathcal G}$ we get $\widetilde{\sigma}^2 \in \Phi(B) \setminus {\mathcal V}(B)$. Since
\begin{align*}
  {\mathcal V}(B) &= Q_{B}^{\perp} = [\{q_{ab}^B \in P_{2}(B) : l(a)l(b) = 0 \in k_{2}(G)\}]^\perp \\
  &= \bigcap \{ (q_{ab}^{B})^\perp \subseteq \Phi(B) : a,b \neq 1 , l(a)l(b) =0\}
\end{align*}
we get $a, b \neq 1$ with $l(a)l(b)=0 \in k_{2}(G)$ and $\widetilde{\sigma}^{2} \in \Phi(B) \setminus (q^{B}_{ab})^\perp$. There are two cases to consider: $a =b$ and $a \neq b$. In the first case $\{a\}$ is a singleton l.i. set and in the second $\{a,b\}$ is a two element l.i. set: consider any well ordered basis $B'= \{a'_{i} : i \in I\}$ {\em such that} $a = a_{i}$ for some $i \in I$ in the first case and, $a=a_{i} , b=a_{j}$ for some $i <j \in I$ in the second case. Now consider the isomorphism of change of basis  $\mu_{B'B} : \mathcal W(B) \overset\cong\rightarrow \mathcal W(B')$ (see Lemma \ref{fixsg2})) and  take $\sigma' = \mu_{B'B}(\widetilde{\sigma}) \in \mathcal W(B')$. Then $\sigma' \in \mathcal W(B')$ is a lifting of $\sigma \in {\mathcal G}$ with respect to the epimorphism $\mathcal W(B') \twoheadrightarrow \mathcal W(B')/{\mathcal V}(B')$ %\underset{\eta_{B'}\to{\overset{\cong}\to\rightarrow}
${\overset{\cong}\rightarrow}$
${\mathcal G}$  (Lemma \ref{lifting-le}, again) and $\sigma'^{2} \in \Phi(B') \setminus \mu_{B'B}[(q^{B}_{a,b})^{\perp}] = \Phi(B') \setminus (q^{B'}_{a,b})^{\perp}$ (see Remark \ref{change-rem}). As $l(a)l(b) =0$, by the proof of the Theorem \ref{standardteo}, we have $(q^{B}_{a_i,a_{i}})^\perp = S'_{i} \cap \Phi(B')$ in first case  and  $(q^{B}_{a_i,a_{j}})^\perp = D'_{ij} \cap \Phi(B')$ in the second case, then $\sigma'^2 \in \Phi(B') \setminus S'_{i}$ (resp. $\sigma'^2 \in \Phi(B') \setminus D'_{ij}$). A straightforward calculation with the group operation in $\mathcal W(B')$ gives $\sigma' \notin  M'_{i} \subseteq \mathcal W(B')$ in the first case and $\sigma' \notin M'_{i}\cup M'_{j} \subseteq \mathcal W(B')$ in the second case, then applying $\mathcal W(B') \twoheadrightarrow {\mathcal G}$ we have $\sigma \notin M_{a_{i}} \subseteq {\mathcal G}$ (resp. $\sigma \notin M_{a_{i}}\cup M_{a_{j}} \subseteq {\mathcal G}$). Now recall that for each $y \in G \setminus\{1\}$ and each $\theta \in {\mathcal G} \setminus \Phi({\mathcal G})$, $\theta \notin M_{y}$ iff $<\{\Phi({\mathcal G}),\theta.\Phi({\mathcal G})\},y> =1$ iff $y \notin \{\Phi({\mathcal G}),\theta.\Phi({\mathcal G})\}^\perp$. Then, since $G$ is a standard pre-special group we have, in both cases, $1 \neq a,b$ , $l(a)l(b) = 0 \in k_{2}(G)$, $a,b \notin \{\Phi({\mathcal G}),\sigma.\Phi({\mathcal G})\}^\perp$ and, in particular, since $G$ is a $k$-stable pre-special group, $1 \in D_G(<a,b>)$ or, equivalently, $<a,b> \equiv <1,ab>$.

\item Recall that for {\em special groups} the maximal saturated subgroups are precisely the index 2 saturated subgroups so the result follows from items (i) and (ii).

\item Let $\Delta \subseteq G$ be a saturated subgroup: as $G$ is a {\em special group} $\Delta = \bigcap \{ \Sigma \subseteq G : \Sigma \in X_{\Delta}\}$ where
$$X_{\Delta} = \{ \Sigma \subseteq G : \Sigma\mbox{ is a maximal saturated subgroup and }\Delta \subseteq \Sigma\}$$ then, by Proposition \ref{pairing-prop},
$$\Delta^{\perp} = \bigvee \{ \Sigma^{\perp} \subseteq {\mathcal G}/{\Phi(\mathcal G)} : \Sigma \in X_{\Delta}\};$$
by item (iii) $\Sigma^{\perp} = \{ \Phi({\mathcal G}), \sigma\Phi({\mathcal G})\}$ for some $\sigma \in {\mathcal G} \setminus \Phi({\mathcal G})$, $\sigma^2 = id$; take $T_{\Sigma} = \Phi({\mathcal G}) \cup \sigma\Phi({\mathcal G})$ then $T_{\Sigma} \subseteq {\mathcal G}$ is a closed (normal) subgroup such that $\Phi({\mathcal G}) \subseteq T_{\Sigma}$ and all elements of $T_{\Sigma} \setminus \{1\}$ are involutions so
$$\bigvee \{ T_{\Sigma} : \Sigma \in X_{\Delta}\} = closure([ \{ T_{\Sigma} : \Sigma \in X_{\Delta}\} ])$$
is a closed subgroup of ${\mathcal G}$ that contains $\Phi({\mathcal G})$ and is (topologically) generated by involutions. Now note that $T_{\Sigma}/\Phi({\mathcal G}) = \{\Phi({\mathcal G}), \sigma\Phi({\mathcal G})\} = \Sigma^{\perp}$ and then
\begin{align*}
  \Delta^\perp &= \bigvee \{ \Sigma^{\perp} \subseteq {\mathcal G}/{\Phi(\mathcal G)} : \Sigma \in X_{\Delta}\} = \bigvee \{ T_\Sigma/\Phi({\mathcal G}) \subseteq {\mathcal G}/{\Phi(\mathcal G)} : \Sigma \in X_{\Delta}\}\\
  &= (\bigvee \{ T_\Sigma \subseteq {\mathcal G} : \Sigma \in X_{\Delta}\})/\Phi({\mathcal G})
\end{align*}
as  $q : {\mathcal G} \twoheadrightarrow {\mathcal G}/{\Phi(\mathcal G)}$ gives an isomorphism of complete lattices between the set of closed subgroups of ${\mathcal G}$ which contains $\Phi({\mathcal G})$  and the set of closed subgroups of ${\mathcal G}/{\Phi(\mathcal G)}$.

Now take $ T \subseteq {\mathcal G}$ a closed subgroup of ${\mathcal G}$  such that $\Phi({\mathcal G}) \subseteq T$ and $T$ is topologically generated by involutions. Write $I_{T} = \{ \sigma \in T : \sigma \in {\mathcal G} \setminus \Phi({\mathcal G}) , \sigma^2 =id\}$ then, for each $\sigma \in I_{T}$, $\sigma \Phi({\mathcal G}) \subseteq T$ , $T_{\sigma} = \Phi({\mathcal G})\cup \sigma\Phi({\mathcal G})$ is a closed (normal) subgroup of ${\mathcal G}$ and
$$T =  closure([ \bigcup \{ T_{\sigma} : \sigma \in I_{T}\}]) = \bigvee \{T_{\sigma} : \sigma \in I_{T}\},$$
also $T_{\sigma}/\Phi({\mathcal G}) = \{\Phi({\mathcal G}),\sigma\Phi({\mathcal G})\}$ and, by item (iv), $(T_\sigma/\Phi({\mathcal G}))^\perp \subseteq G$ is a maximal saturated subgroup of $G$. Then we have
\begin{align*}
    (T/\Phi({\mathcal G}))^\perp &= ( (\bigvee \{T_{\sigma} : \sigma \in I_{T})/\Phi({\mathcal G}) )^\perp = (\bigvee \{T_{\sigma}/\Phi({\mathcal G}) : \sigma \in I_{T}\})^\perp \\
    &= \bigcap \{(T_{\sigma}/\Phi({\mathcal G}))^\perp : \sigma \in I_{T}\}
\end{align*}
which is a saturated subgroup of $G$.
\end{enumerate}
\end{proof}

\begin{teo} \label{reduced-teo}
 Let $G$ be a standard special group. Are equivalent
 \begin{enumerate}[i -]
     \item $G$ is "Pythagorean" or "almost reduced"\footnote{I.e. for all $a \in G$, $<a,a> \equiv <1, 1>$ iff $a =1$, but eventually $-1 = 1$.}.
     \item $\mbox{Gal}(G)$ is generated by involutions.
 \end{enumerate}
\end{teo}

\begin{proof}
Note that the unique non-formally real Pythagorean special group  (equivalently, $-1 \neq 1$) is $G=\{1\}$ and thus $Gal(G) = \{1\}$.

$(i) \Rightarrow (ii)$: The hypothesis means that $\{1\}\subseteq G$ is a saturated subset of $G$, then by item (iv) of the previous Proposition, $Gal(G)$ is generated by involutions.

$(ii) \Rightarrow (i)$: It follows the hypothesis that there is no continuous epimorphism ${\mathcal G}/\Phi({\mathcal G}) \twoheadrightarrow {\mathbb Z_4}$. Since $G$ is standard SG, for all $a\in G \setminus \{1\}$, $l(a)l(a) \neq 0 = l(1)l(1)$ and, since $G$ is in particular $k$-stable, then for all $a \in G \setminus\{1\}$, is not the case $<a,a>\equiv <1,1>$, that is: $G$ is Pythagorean.
\end{proof}

\begin{rem}
    Another Galois theoretic characterization of the  Pythagoreaness of $G$ is
    $$\Phi(\mbox{Gal}(G))=[\mbox{Gal}(G),\mbox{Gal}(G)].$$
\end{rem}

%formally realthis  follows directly from item (iv) of the  Theorem above.  If $-1 =1$, since $G$  is a Pythagorean SG, then

% thus it is trivilly generated : Note that $\{1\}$ is a saturated subgroup iff $\{1\}$ corresponds with a closed subgroup generated by involutions: and as ${\mathcal G}/\Phi({\mathcal G}) = \{1\}^{\perp}$ then ${\mathcal G}$ corresponds with $\{1\}$  corresponds  is generated by involutions $(ii) \Rightarrow (iii)$

%${\mathcal G}$ is a ${\mathcal C}$-group then $\Phi({\mathcal G}) \subseteq Center({\mathcal G})$ (as $\Phi({\mathcal G}) = {\mathcal G}^2 and [{\mathcal G}^2,{\mathcal G}]=1$)

%Como ? Todo conjunto gerador contem um subconjunto que converge a 1: Ribes pag 44, VER LIVRO KOCH, todo quadrado cescente e produto de quadrado de imnvolicions and commutatdos ; ...

\begin{teo}
Let $G$ be a standard special group. Consider the following
\begin{enumerate}[i-]
\item  $G$ is not formally real.
\item Every involution is is $\Phi(Gal(G))$.
\item  Every involution in ${Gal(G)}$ is central.
\end{enumerate}
Then $(i) \Rightarrow (ii) \Rightarrow (iii)$  and if $card({Gal(G)}) > 2 $, then all are equivalent.
\end{teo}

\begin{proof}
 By Proposition \ref{encoding-pr}(iii) $G$ is formally real iff there is an involution $\sigma \in {Gal(G)} \setminus \Phi({Gal(G)})$, so we get $(i) \Rightarrow (ii)$. As ${Gal(G)}$ is a ${\mathcal C}$-group we have $[\sigma^2,\tau] =1$ and, since  $\Phi({Gal(G)}) = {Gal(G)}^2$ (pro-2-group), then $\Phi({Gal(G)}) \subseteq center({Gal(G)})$ so $(ii) \Rightarrow (iii)$. Now suppose  $card({Gal(G)}) >2$: to prove $(iii) \Rightarrow (i)$ let us assume $G$ formally real and note that any involution $\notin \Phi({Gal(G)})$ {\em is not in} the center of ${Gal(G)}$.
\end{proof}

\begin{comment}
\begin{prop}
Let $G$ be a k-stable special group. Write $\mathcal G=\mbox{Gal}(G)$.
\begin{enumerate}[i -]
    \item If $\sigma\in\mathcal G\setminus\Phi(\mathcal G)$ with $\sigma^2=id$ then $\{id,\sigma/\Phi\}^{\perp} \subseteq G$ is a maximal saturated subgroup.

    \item If $G$ is a {\em pre-special} group and $\sigma \in \mathcal G -\Phi(G)$, then $\sigma^{2} \neq id$ (so $\sigma^{4} =id$)
and $\{id,\Sigma/Phi\}^{\perp} \subseteq G$ is {\em not} a saturated subgroup.

    \item If $G$ $k$-stable {\em pre-special} group then the classes of involutions $\sigma \in \mathcal G -\Phi(\mathcal G)$ corresponds to orderings (maximal saturated subgroups) via $H_{\sigma} = \{id, \sigma\}, \sigma^{2} =id$, and both are closed sets.

    \item If $G$ is $k$-stable {\em pre-special} group then we have and anti-isomorphism of complete lattices between the  set of $\{ \Delta \subseteq G : \Delta$ is a saturated subgroups of $G\}$ and the $\{T \subseteq \mathcal G : T $ is a closed subgroup of $\mathcal G$ (topologically) generated by involutions such that $\Phi(\mathcal G \subseteq T\}$.
\end{enumerate}
\end{prop}
\end{comment}

\section{The functorial behavior of $Gal$ and SG-cohomology}

We have developed the theme "basis change induced {\em isomorphisms}" in the general context of {\em pre-Special Groups}, as the fundamental step to get a {\em single} Galois group of a pre-special group. In this final section, we analyze some functorial behavior of the $Gal$ construction of $SG$-theory and provide the first steps to a Galoisian cohomology for the $SG$-theory, in an attempt to complete the "Milnor scenario" of Igr's (\cite{milnor1970algebraick}) in abstract theories of quadratic forms.

\subsection{From PSG to Galois groups}

\begin{rem} \label{construPSGGAL-ct} {\bf Construction:}

Let $f: G \rightarrow G'$ a $pSG$-homomorphism of pre-special groups.

 Let $B_1'=\{a'_{k} : k \in I_1'\}$ be an well ordered basis of $f[G]$ and extends it to $B'=\{a'_{k} : k \in I'\}$, an well ordered basis of $G'$. Now  {\em select} $a_{k} \in f^{-1}[\{a'_{k}\}] , k \in I'_1$.  Then the set $B_1 = \{a_{i} : i \in I'_1\} \subseteq G$ is linearly independent, now complete this to basis of $G$, $B= \{a_i : i \in I\}$: we just need to glue  a well ordered basis of $ker(f)$.

We have some induced functions:

(0) ${f}^{0}_{B,B'} : B' \rightarrow {\mathcal W}(B)$  is  such that ${f}^{0}_{B,B'}(a'_{k}) = a_k$, if $a'_k \in B'_1$ and ${f}^{0}_{B,B'}(a'_{k}) = 1$, if $a'_k \in B' \setminus B'_1$.

(1) $f^{(1)}_{B,B'} : B \rightarrow P_{2}(B')$ is  such that ${f}^{1}_{B,B'}(a_{k}) = a'_k$, if $a_k \in B_1$ and ${f}^{1}_{B,B'}(a_{k}) = 0$, if $a_k \in B \setminus B_1$.

(2) $f^{(2)}_{B,B'} : \{(a_i, a_j) \in B \times B : i, j \in I, i \leq j\} \rightarrow P_{2}(B')$  is  such that ${f}^{2}_{B,B'}(a_i, a_j) = a'_i.a'_j$, if $(a_i, a_j) \in  B_1 \times B_1$ and  ${f}^{2}_{B,B'}(a_i, a_j) =   0$, if  $(a_i, a_j) \in  B \times B \setminus  B_1 \times B_1$.
\end{rem}

%Then we get an induced continuous homomorphism $f^{\star}_{B,B'} : {\mathcal W}(B')/{\mathcal V}(B') \rightarrow {\mathcal W}(B)/{\mathcal V}(B)$.

%; besides the family of those homomorphisms are compatible (=functorial) in some (weak) sense.

%We break the construction and proofs in two Lemmas:

%There is a natural way to

Keeping the notation above, we have

\begin{prop} \label{GaltoSG-pr}
The function ${f}^{0}_{B,B'} : B' \rightarrow {\mathcal W}(B)$ induces a continuous homomorphism $\bar{f}_{B,B'} : {\mathcal W}(B')/{\mathcal V}(B') \rightarrow {\mathcal W}(B)/{\mathcal V}(B)$.
\end{prop}

\begin{proof}
It follows from Proposition \ref{fixms3} and the definition of ${f}^{0}_{B,B'} : B' \rightarrow {\mathcal W}(B)$, that its image converges to $1 \in {\mathcal W}(B)$. Thus, by the universal property ${\mathcal W}(B')$ (Theorem \ref{fixms4}),  ${f}^{0}_{B,B'}$ extends uniquely to a continuous homomorphism  of pro-2-groups  $\hat{f}^0_{B,B'} : {\mathcal W}(B') \rightarrow {\mathcal W}(B)$.

Now, $f : G \to G'$ also induces a $\mathbb{Z}_2$-module homomorphism $\hat{f}^{(2)}_{B,B'} : P_{2}(B) \rightarrow P_{2}(B')$: this is just the unique $\mathbb Z_2$-linear extension of the induced map  $f^{(2)}_{B,B'} : \{(a_i, a_j) \in B \times B : i, j \in I, i \leq j\} \rightarrow P_{2}(B')$. Moreover, since $k_*(f) : k_*(G) \to k_*(G')$ is an Igr-morphism, then for each $a, b \in G$ such that  $l(a)l(b) = 0 \in k_2(G)$, we have $l(fa).l(fb) = 0 \in k_2(G')$. From this we obtain that $\hat{f}^{(2)}_{B,B'}(q^B_{a,b}) = q^{B'}_{fa,fb}$ (see Proposition \ref{fixsg1}) and, therefore, $f^{(2)}_{B,B'}[Q(B)] \subseteq Q(B')$.

For each $a,b \in G$ and $\sigma' \in {\mathcal W}(B')$, we have
$$<\hat{f}^{(0)}_{B,B'}(\sigma') , q^{B}_{a,b}> = <\sigma', \hat{f}^{(2)}_{B,B'}(q^{B}_{a,b}> = <\sigma',  q^{B'}_{fa,fb}> .$$

%, gerados, produto, densidade

Thus $\hat{f}^{(0)}_{B,B'}[{\mathcal V}(B')] \subseteq {\mathcal V}(B)$ and then, $\hat{f}^0_{B,B'} : {\mathcal W}(B') \rightarrow {\mathcal W}(B)$ induces a unique continuous homomorphism of pro-2-groups $\bar{f}_{B,B'} : {\mathcal W}(B')/{\mathcal V}(B') \rightarrow {\mathcal W}(B)/{\mathcal V}(B)$.

%(v) ``carater funtorial''weak  $G \rightarrow G' \rightarrow G''$ $B'' \mapsto B'_{f'} \cup K'_{f'} \mapsto B_{f}\cup K(f) , Kf'\circ f \subseteq K_{f}$
\end{proof}

%\begin{rem} When $f : G \rightarrow G'$ is an {\em injective} $pSG$-morphism we can identify $\hat{f}^0_{B,B'} : {\mathcal W}(B') \rightarrow {\mathcal W}(B)$ and $\bar{f}_{B,B'} : {\mathcal W}(B')/{\mathcal V}(B') \rightarrow {\mathcal W}(B)/{\mathcal V}(B)$ with   {\em projections}, i.e., they are surjective morphisms of pro-2-groups.
%\end{rem}

%\begin{rem}
%Remark ' dependente da base $G \sigma {\mathbb Z_2} \hookrightarrow G'$ ``ker(f) nao contido em ${\mathcal V}(B)$''
%\end{rem}

%\begin{prop} $F : G \rightarrowtail G'$ injective (ESTABILIZAR monoMORPHISMOS POR FATORACOES)
%(1) $f*$ and $f\star$ are epimorphisms also $B'i \twoheadrightarrow Bi$ is compatible with basis change (ESTABILIZAR EPIMORPHISMOS POR FATORACOES)
%(2)' carater funtorial para morphismos injectives  condicao necessaria e suficiente para obter flecha entre limites indutivos (morp de conexao sao isos)
%\end{prop}

\begin{prop}
Let $f : G \rightarrow G'$ be an injective $pSG$-morphism.
\begin{enumerate}[i -]
    \item Then $\hat{f}^0_{B,B'} : {\mathcal W}(B') \rightarrow {\mathcal W}(B)$ and $\bar{f}_{B,B'} : {\mathcal W}(B')/{\mathcal V}(B') \rightarrow {\mathcal W}(B)/{\mathcal V}(B)$  are surjective morphisms of pro-2-groups (thus they can be identified  with   {\em projections}).
 \item Let  $f' : G'' \rightarrow G$ is an injective $pSG$-morphism. If $B'$ is an well ordered basis of $G'$ obtained by successive extensions of an well ordered basis $B'_2$ of $f\circ f'[G'']$ to $B'_1 $, an  well ordered basis  of $f[G] \supseteq f\circ f'[G'']$, then applying the construction  above described,  we obtain $\hat{f}^0_{B'',B'} =  \hat{f}^0_{B'',B} \circ \hat{f}^0_{B,B'} : {\mathcal W}(B') \rightarrow {\mathcal W}(B'')$ and $\bar{f}^0_{B'',B'} =  \bar{f}^0_{B'',B} \circ \bar{f}^0_{B,B'} :   {\mathcal W}(B')/{\mathcal V}(B') \rightarrow {\mathcal W}(B'')/{\mathcal V}(B'')$.
\end{enumerate}
\end{prop}

\begin{proof}
$ $
\begin{enumerate}[i -]
    \item By the injectivity hypothesis, we have ${f}^0_{B,B'}[B'] = B \cup \{1\} \subseteq {\mathcal W}(B)$, thus $\hat{f}^0_{B,B'} : {\mathcal W}(B') \rightarrow {\mathcal W}(B)$ is a continuous function with dense image from a compact space into a Hausdorff space. Therefore $\hat{f}^0_{B,B'}$ and $\bar{f}^0_{B,B'} $ are surjective continuous homomorphisms.

    \item It follows from a straightforward calculation that ${f}^0_{B'',B'} =  \hat{f}^0_{B'',B} \circ {f}^0_{B,B'}$. Therefore, the uniqueness of extensions and the homomorphism theorem guarantees that $\hat{f}^0_{B'',B'} =  \hat{f}^0_{B'',B} \circ \hat{f}^0_{B,B'}$ and $\bar{f}^0_{B'',B'} =  \bar{f}^0_{B'',B} \circ \bar{f}^0_{B,B'}$.
\end{enumerate}
\end{proof}

\subsection{From Galois Groups to PSG}

Let $G$ be a pre-special group an denote $\mathcal G = Gal(G)$. We have seen in Proposition \ref{fixhugo3} that there is a {\em canonical} isomorphism $\phi_{G} : {\mathcal G}/\Phi({\mathcal G}) \overset{\cong}\rightarrow Hom(G,{\mathbb Z_2})$ so we get a ``perfect pairing'' $\hat{\phi}_{G} : {\mathcal G}/\Phi({\mathcal G}) \times G \rightarrow {\mathbb Z_2}$ and there is also  a {\em canonical} bijection  $G  \cong \{T \subseteq {\mathcal G} : T$ is a closed normal subgroup of index $\leq 2\}$.

We will  explain now the term ``canonical'' employed, starting with the following

\begin{lem}
$ $
\begin{enumerate}[i -]
    \item Let $G, G'$ be pre-special groups  Then each continuous homomorphism $\theta : Gal(G') \ \rightarrow \ Gal(G)$ induces a $\mathbb{Z}_2$-module homomorphism $\check{\theta} : G \to G'$.

\item  The association above, $\theta \mapsto \check{\theta}$, determines a contravariant functor from the category from all pairs $(G, Gal(G))$, $G$ a pre-special group, and continuous homomorphisms, into the category of ${\mathbb Z_2}$-modules.

%Galois groups ofpre-special groups and continuous homomorphism  to the category of ${\mathbb Z_2}$-vector spaces.  in a FUNCTORIAL way (because diagrams above comute)
%(iii) The association $\theta^{\star}$ can be identifyied with the arrow $\theta^{-1} : \{T \subseteq {\mathcal G} : T$ is a closed normal subgroup of index $\leq 2\} \rightarrow   \{T' \subseteq {\mathcal G}' : T'$ is a closed normal subgroup of index $\leq 2\}$ : $T \mapsto \theta^{-1}[T]$
\end{enumerate}

\end{lem}
\begin{proof}
$ $
\begin{enumerate}[i -]
    \item We have a ${\mathbb Z_2}$-homomorphism $\theta^{*} : Homcont(Gal(G), {\mathbb Z_2}) \to Homcont(Gal(G'), {\mathbb Z_2})$, $ \mu \mapsto \mu\circ \theta $.   By Proposition \ref{fixhugo3}(iii), we have ${\mathbb Z_2}$-isomorphisms $\psi_G, \psi_{G'}$. Combining the information we define the ${\mathbb Z_2}$-homomorphism  $\check{\theta} := \psi_{G'}^{-1} \circ\theta^*\circ \psi_G  : G \ \rightarrow \ G'$.
    \item Note that  $id_{\mathcal G}^* = id$, thus $\check{id}_{\mathcal G} = id_G$. Let $\theta' : Gal(G'') \to Gal(G')$ be a continuous homomorphism. Then  we have $(\theta \circ \theta')^* = {\theta'}^* \circ \theta^*$, thus $(\theta \circ \theta')\check{ } = \check{\theta'} \circ \check{\theta}$.
    \end{enumerate}
    \end{proof}
 %By Proposition \ref{fixms5}, we have $\theta[\Phi({\mathcal G}')] \subseteq \Phi({\mathcal G})$, thus, by the Theorem of Homomorphism there is a unique pro-2-group morphism $\bar{\theta} : {\mathcal G}'/\Phi({\mathcal G}') \to {\mathcal G}/\Phi({\mathcal G})$$ such that $\phi_{G}\circ \overline{\theta} = \phi_{G'} \circ Hom(\theta^{\star})$ where $\overline{\theta} : {\mathcal G}'/\Phi({\mathcal G}') \rightarrow {\mathcal G}/\Phi({\mathcal G})$ is the quotient of  $\theta : {\mathcal G}' \rightarrow {\mathcal G}$ and  $Hom(\theta^{\star}) : Hom(G',{\mathbb Z_2}) \rightarrow Hom(G,{\mathbb Z_2})$ is the arrow $(?)\circ\theta^{\star}$ , as in diagrams below.
%$\phi {\mathcal W} quot e \phi' {\mathcal W}/{\mathcal V} quot'$ logo $quo' \cong quo \cong G^* $
%In that way we get, for each contynuos homomorphism $\theta : gal(G') \rightarrow {\mathcal G}$ an ${\mathbb Z_2}$-homomorphism $\theta^{\star} L: G \rightarrow G'$ in a FUNCTORIAL way (because diagrams above comute)

\begin{rem}
Note that for  any surjective continuous homomorphism $\theta : {\mathcal G}' \rightarrow {\mathcal G}$, we have that $\check{\theta} : G \to G'$ is an injective ${\mathbb Z_2}$-homomorphism.
 \end{rem}

This suggest that the (sub)category of Galois groups and continuous epimorphisms is the ``right'' domain category of Galois groups. We have the following:

\begin{prop}
The functor described in the Lemma above restricts to a functor from the  subcategory of Galois groups of standard pre-special groups (Definition above \ref{pSGstandard-def}) and continuous epimorphisms to the category of standard pre-special groups and injective $qSG$-morphisms (i.e., the group homomorphisms that preserves $\equiv$, but that eventually does not preserves -1).
\end{prop}
\begin{proof}
Assume that $G$ is a $k$-stable pre-special group and that $G'$ is a standard pre-special group, we will prove that $\check{\theta}$ is a injective $qSG$-homomorphism from $G$ to $G'$.

Since  $\check{\theta} : G \to G'$ is a group homomorphism, it   is enough to show that, for each $a,b \in G\setminus\{1\}$, $1 \in D_{G}<a,b> \Rightarrow 1' \in D_G<\check{\theta}(a), \check{\theta}(b)>$ and, since $G, G'$ are $k$-stable pre-special group, this is equivalent to show $l(a)l(b) = 0 \in k_2(G) \Rightarrow l(\check{\theta}(a))l(\check{\theta}(b)) = 0 \in k_{2}(G')$. Now we have 2 cases to consider:

{\underline{$a =b$}} : Since $G$ is a $k$-stable pre-special group then, by Theorem \ref{standardteo}(i), there exists $S \subseteq {\mathcal G}$ a closed normal subgroup with ${\mathcal G}/S \cong {\mathbb Z_4}$ and $S \subseteq M_{a}$. As $\theta : {\mathcal G}' \twoheadrightarrow {\mathcal G}$ is an {\em surjective} continuous homomorphism we have that the quotient map $\theta_{S} : {\mathcal G}'/\theta^{-1}[S] \rightarrow {\mathcal G}/S$ is an {\em isomorphism} so $\theta^{-1}[S] \subseteq {\mathcal G}'$ is a closed normal subgroup such that ${\mathcal G}'/\theta^{-1}[S] \cong {\mathcal G}/S \cong {\mathbb Z_4}$ and $\theta^{-1}[S] \subseteq \theta^{-1}[M_{a}] = M'_{\theta^\star(a)}$, where this last equality holds  by the bijections in items (ii) and (iii) in Proposition \ref{fixhugo3} and the definition of $\check{\theta}$.  Now, since $G'$ is a standard pre-special group, we have $l(\check{\theta}(a))l(\check{\theta}(a))=0 \in k_{2}(G')$

{\underline {$a \neq b$}} : As  $\check{\theta}$ is an {\em injective} ${\mathbb Z_2}$-homomorphism, $\check{\theta}(a) \neq \check{\theta}(b)$.  Since $G$ is a $k$-stable pre-special group then, by Theorem \ref{standardteo}(ii), there exists $D \subseteq {\mathcal G}$ a closed normal subgroup with ${\mathcal G}/D \cong \mathbb{D}_4$, $D \subseteq M_{a}\cap M_{b}$ and $M_{ab}/D \cong {\mathbb Z_4}$. As ${\theta} : {\mathcal G}' \twoheadrightarrow {\mathcal G}$ is an {\em surjective} continuous homomorphism we have that the quotient map $\theta_{D} : {\mathcal G}'/\theta^{-1}[D] \rightarrow {\mathcal G}/D$ is an {\em isomorphism} so $\theta^{-1}[D] \subseteq {\mathcal G}'$ is a closed normal subgroup such that ${\mathcal G}'/\theta^{-1}[D] \cong {\mathcal G}/D \cong {\mathbb D_4}$ and $\theta^{-1}[D] \subseteq \theta^{-1}[M_{a} \cap M_{b}] = M'_{\check{\theta}(a)} \cap M'_{\check{\theta}(b)}$. Now we  check that $M'_{\check{\theta}(a) \check{\theta}(b)}/\theta^{-1}[D] \cong {\mathbb Z_4}$:  as $\theta$ is an epimorphism $\theta[theta^{-1}[M_{ab}]] = M_{ab}$ so, as $M'_{\check{\theta}(a) \check{\theta}(b)} = M'_{\check{\theta}(ab)} = \theta^{-1}[M_{ab}]$ (because $\check{\theta}(a) \neq \check{\theta}(b)$ and there are exactly three maximal above $M'_{\check{\theta}(a)} \cap M'_{\check{\theta}(b)} $), we have an epimorphism $\theta_| : M'_{\check{\theta}(ab)} \twoheadrightarrow M_{ab}$ thus $ker(\theta_|) = \theta^{-1}[D] \subseteq M'_{\check{\theta}(ab)}$ and the quotient map  $\theta_{|{D}} : M'_{\check{\theta}(ab)}/\theta^{-1}[D] \rightarrow M_{ab}/D$ is an {\em isomorphism} so $M'_{\check{\theta}(ab)}/\theta^{-1}[D] \cong M_{ab}/D \cong {\mathbb Z_4}$. Now, since $G'$ is a standard pre-special group, we have $l(\check{\theta}(a))l(\check{\theta}(b))=0 \in k_{2}(G')$.
\end{proof}

%Lemas M, S, D
%Lema M: $G- \{1\} \cong \{$maximais clopen $H\}$ , $G \cong \{$clopens de indice 1 ou 2$\}$
%Lema M diz $\phi' = \phi/{\mathcal V}$
%$\phi {\mathcal W} quot e \phi' {\mathcal W}/{\mathcal V} quot'$ logo
%$quo' \cong quo \cong G^*$

%In that way we get, for each contynuos homomorphism $\theta : gal(G') \rightarrow {\mathcal G}$ an ${\mathbb Z_2}$-homomorphism $\theta^{\star} L: G \rightarrow G'$ in a FUNCTORIAL way (because diagrams above comute)

%Lema S,D quase obtemos psg -morphiosms o que faslta e $G\cong G_(k$

%FALTA PROVAR AS VOLTAS DOS LEMAS S,D: VER HOECHS. ,FROLICH JOURNAL CRELLE 1968, 1985

%ESTABILIZAR EPIMORPHISMOS por fatoracoes

%Propositions similars to MS

%(1) $f injetora$ entre $k- staveis$ : $f** = f$ como grupos pre speciais

%(2) $aut({\mathcal G})/N \cong qsgaut(G)$ $N = \{\psi \in aut({\mathcal G} : psi(sigma).sigma^{-1} \in \Phi({\mathcal G})\}$

%Fiish the section with the remark Witt ring

%(3) the results in proposition ? e ? works (universal binary formas..)

\subsection{Towards  a galoisian cohomology for $SG$-theory}

Let $G$ be a standard pre-special group. Since $\mathcal G=\mbox{Gal}(G)$ is a profinite group, the Galois Cohomology is available for this subclass of pre-special groups. In particular, there is the graded cohomology ring $H^*(G):= H^*(\mathcal G, \{\pm 1\})$ where  $\mathcal G$ act trivially on $\mathbb F_2=\mathbb Z/2\mathbb Z$. Therefore, at least some parts of Milnor's scenario for containing 3 graded rings related to quadratic forms theory of fields (with $char \neq 2$) is available for (standard) special groups: $W_*(G), k_*(G),$ and $ H^*(G)$.

The result above just provides the initial step to establish cohomological methods in SG-theory.

%in this section we describe, for any pre sprecial group G, a Milnor like canonical arrow from the mod 2 k-theory graduated ring of G to the  graduated ring of cohomology of G: $h(G) : k_{\ast}(G) \rightarrow H^{\ast}(G)$

%Lemma: for each $a, b \in G$ are equivalent: $l(a)l(b)=0 \in k_{2}(G) \iff (a)\cup (b) = 0 \in H^{2}(G) = H^{2}({\mathcal G},{\mathbb Z_2})$

%These  generalizes some results in \cite{AKM} to the context of (pre)special groups where they prove that the cohomology ring $H^{\ast}(gal(F^{(3)}|F),{\mathbb Z_2})$ contains the cohomology ring $H^{\ast}(gal(F^{s}|F), {\mathbb Z_2})$ as its subring generated by cup products of level 1 elements

%OS LEMAS M,S,D SAO DECISIVOS PARA MOSTRAR QUE H* EH IGR~ E EXISTE SG-MORF  G to H1(G)
%(INJETIVIDADE : VER MINC GAO)
%UITILIZAR PROPRIEDADE UNIVERSAL de K* para obter flecha canonica

\begin{teo}
    As in the field case, consider $\mathbb{Z}_2 \cong \{\pm 1\}$ as a discrete $Gal(G)$-module endowed with the trivial action, i.e., $\sigma.a = a$, for all $\sigma \in Gal(G)$ and $a \in \mathbb{Z}_2 $. Then $H_*(G) := H_*(Gal(G), \mathbb Z_2)$,  is a graded ring endowed with the cup product.   Moreover, there are  canonical isomorphisms of pointed $2$-groups: \\(i) $(\mathbb{Z}_2, -1) \cong (H^0(G), (-1))$;\\
    (ii) $(G,-1) \cong (H^1(G), (-1))$.
\end{teo}
    \begin{proof} We write $\mathcal G :=Gal(G)$.
     Just recall that:

     $H^{0}(\mathcal G , \mathbb{Z}_2) =  (\mathbb{Z}_2)^{\mathcal G }= Fix(\mathbb{Z}_2) = {\mathbb Z_2}$, since $\mathcal G $ is acting trivially on $\mathbb Z_2$.

For $H^1(\mathcal G,\mathbb Z_2) := CrossedHom(\mathcal G, \mathbb{Z}_2)/principalCrossedHom(\mathcal G, \mathbb{Z}_2)$, since $\mathcal G$ is acting trivially on $\mathbb Z_2$, we get
\begin{align*}
   principalCrossedHom(\mathcal G, \mathbb{Z}_2) &:=\
  \mbox{Im}(\overline\partial_1) \\
  &:=\{x:\mathcal G\rightarrow \mathbb Z_2:x=\overline\partial_1a\mbox{ for some }a\in \mathbb F_2\} \\
  &=\{x:\mathcal G\rightarrow \mathbb Z_2:\mbox{ there exist }\mathbb F_2\in \mathbb F_2\mbox{ such that }x(\sigma)=\sigma a-a\mbox{ for all }\sigma\in \mathcal G\} \\
  &=\{x:\mathcal G\rightarrow \mathbb Z_2:\mbox{ there exist }a\in \mathbb F_2\mbox{ such that }x(\sigma)=0\mbox{ for all }\sigma\in \mathcal G\}&\\
  &=\{0\};
\end{align*}
and
\begin{align*}
     CrossedHom(\mathcal G, \mathbb{Z}_2) &:=\
    \mbox{Ker}(\overline\partial_2) \\
    &:=\{x:\mathcal G\rightarrow \mathbb Z_2:x\mbox{ is continuous and }x(\sigma\tau)=\sigma x(\tau)+x(\sigma)\mbox{ for all }\sigma,\tau\in \mathcal G\} \\
    &=\{x:\mathcal G\rightarrow \mathbb Z_2:x\mbox{ is continuous and }x(\sigma\tau)=x(\tau)+x(\sigma)\mbox{ for all }\sigma,\tau\in \mathcal G\} \\
    &={Homcont}(\mathcal G,\mathbb Z_2).
\end{align*}

Therefore $H^{1}(\mathcal G, \mathbb{Z}_2) = Homcont(\mathcal G,\mathbb{Z}_2)/\{0\} \cong    Homcont(\mathcal G, \mathbb{Z}_2) = Homcont(Gal(G), \mathbb{Z}_2)$.

On the other hand, by Proposition \ref{fixhugo3}(iii) $\psi_G : G \overset\cong\to Homcont(Gal(G), \mathbb{Z}_2)$ as $\mathbb{Z}_2$-modules and, $-1 \in G$ corresponds to a open subgroup of $Gal(G)$ with index $\leq 2$, that corresponds to $(-1)$ in $Homcont(Gal(G), \mathbb{Z}_2)$.
\end{proof}

It is natural ask if the Igr $H^*(Gal(G), \{\pm1\})$ is in the subcategory Igr$_h$: this depends of an analysis and more explicit description of $H^2(Gal(G), \{\pm1\})$. In particular, we will need to analyze the relationship between the equations $l(a)l(b)=0 \in k_{2}(G)$ and $ (a)\cup (b) = 0 \in  H^{2}({Gal(G)},{\mathbb Z_2})$, for $a, b \in G$, in a standard pre-special group $G$. Related to this question is the existence of a Milnor like canonical arrow from the mod 2 k-theory graduated ring of $G$ to the  graduated ring of cohomology of G: $h(G) : k_{\ast}(G) \rightarrow H^{\ast}(G)$.

These  generalizes some results in \cite{adem1999cohomology} to the context of (pre)special groups where they prove that the cohomology ring $H^{\ast}(Gal(F^{(3)}|F),{\mathbb Z_2})$ contains the cohomology ring $H^{\ast}(Gal(F^{s}|F), {\mathbb Z_2})$ as its subring generated by cup products of level 1 elements. Therefore, it could be interesting also analyze the properties of the  sub inductive graded generated in level 1 of the Igr $H^*(Gal(G), \{\pm1\})$, that is possibly a member of the subclass Igr$_+$.

\section{Final remarks and future works}

We list below the next steps towards the development of a Galois theory for the subclass of all standard pre-special hyperfields  (or, equivalently, standard pre-special groups), that contains the class of all special hyperfields (or special groups).

\begin{enumerate}

\item Analyse closure of the class of standard pSG under constructions like quotients, (finite) products, directed colimits, and extensions. Moreover, try to provide descriptions of the Galois under some of the main constructions of special groups.

    \item  Develop a theory of quadratic extensions for some subcategory of the category of pre-special hyperfields, envisaging to obtain, for each standard pre-special hyperfield $F$, a description of $gal(F)$ "from below" as a limit of groups of certain low degree extensions of the hyperfield $F$.

     \item Based on the former item, for each standard pre-special hyperfield $F$, analyze if $H^{*}(gal(F),\{\pm 1\})$ is an IGR and determine, a natural IGR-morphism ($h_{n}(F) : k_{n}(F) \rightarrow H^{n}(gal(F),\{\pm 1\}))_{n \in \mathbb{N}}$).

      \item Determine sufficient conditions for $h_n(F): k_{n}(F) \to H^{n}(gal(F),\{\pm 1\})$ be injective, $n \in \mathbb{N}$.

\item Describe a category of "pointed" pro-2-groups $({\mathcal G}, \mathfrak{m})$, where $\mathfrak{m} \subseteq {\mathcal G}$ is a maximal (cl)open subgroup, that contains the Galois groups of special hyperfields and analyze if the relationship described in \cite{efrat2017galoiscohomology}  $({\mathcal G}, \mathfrak{m}) \mapsto H^{*}(({\mathcal G}, \mathfrak{m}),\{\pm 1\})$ is part of a contravariant adjunction between this category and category of IGRs.

%[{\bf EM}] {EM} I. Efrat, J. Min\'a\v c, {\em Galois groups and cohomological functors}, Transactions of AMS {\bf 369} n.4 (2017), 2697-2720.

      \item Apply cohomological methods to describe possible obstructions for the validity of Marshall's signature conjecture for all reduced special groups (or, equivalently, for real reduced hyperfields).

\
\end{enumerate}

\bibliographystyle{plain}
\bibliography{one_for_all}

\begin{thebibliography}{10}

\bibitem{adem1999cohomology}
A.~Adem, D.~B. Karagueuzian, and J.~Min\'a\v{c}.
\newblock On the cohomology of {G}alois groups determined by {W}itt rings.
\newblock {\em Advances in Mathematics}, (148):105–--160, 1999.

\bibitem{roberto2021quadratic}
Kaique~Matias de~Andrade~Roberto, Hugo~Rafael de~Oliveira~Ribeiro, and
  Hugo~Luiz Mariano.
\newblock Quadratic structures associated to (multi) rings.
\newblock {\em Categories and General Algebraic Structures}, 16(1):105--141,
  2022.

\bibitem{roberto2021ktheory}
Kaique~Matias de~Andrade~Roberto and Hugo~Luiz Mariano.
\newblock K-theories and free inductive graded rings in abstract quadratic
  forms theories.
\newblock {\em Categories and General Algebraic Structures}, 17(1):1--46, 2022.

\bibitem{roberto2021ACmultifields1}
Kaique~Matias de~Andrade~Roberto and Hugo~Luiz Mariano.
\newblock On superrings of polynomials and algebraically closed multifields.
\newblock {\em Journal of Pure and Applied Logic}, 9(1):419--444, 2022.

\bibitem{roberto2022graded}
Kaique~Matias de~Andrade~Roberto and Hugo~Luiz Mariano.
\newblock {I}nductive {G}raded {R}ings associated to {Q}uadratic {M}ultirings.
\newblock {\em in preparation}, 2023.

\bibitem{ribeiro2016functorial}
Hugo~Rafael de~Oliveira~Ribeiro, Kaique~Matias de~Andrade~Roberto, and
  Hugo~Luiz Mariano.
\newblock Functorial relationship between multirings and the various abstract
  theories of quadratic forms.
\newblock {\em S{\~a}o Paulo Journal of Mathematical Sciences}, 16:5--42, 2022.

\bibitem{dickmann1998quadratic}
Maximo Dickmann and Francisco Miraglia.
\newblock On quadratic forms whose total signature is zero mod $2^n$: Solution
  to a problem of {M}. {M}arshall.
\newblock {\em Inventiones mathematicae}, 133(2):243--278, 1998.

\bibitem{dickmann2000special}
Maximo Dickmann and Francisco Miraglia.
\newblock {\em Special groups: Boolean-theoretic methods in the theory of
  quadratic forms}.
\newblock Number 689 in Memoirs AMS. American Mathematical Society, 2000.

\bibitem{dickmann2003lam}
Maximo Dickmann and Francisco Miraglia.
\newblock Lam's conjecture.
\newblock In {\em Algebra Colloquium}, pages 149--176. Springer-Verlag, 2003.

\bibitem{dickmann2006algebraic}
Maximo Dickmann and Francisco Miraglia.
\newblock Algebraic k-theory of special groups.
\newblock {\em Journal of Pure and Applied Algebra}, 204(1):195--234, 2006.

\bibitem{efrat2017galoiscohomology}
Ido Efrat and J{\'a}n~Min\'a\v c.
\newblock Galois groups and cohomological functors.
\newblock {\em Transactions of AMS}, 369(4):2697--2720, 2017.

\bibitem{marshall2006real}
Murray Marshall.
\newblock Real reduced multirings and multifields.
\newblock {\em Journal of Pure and Applied Algebra}, 205(2):452--468, 2006.

\bibitem{milnor1970algebraick}
John Milnor.
\newblock Algebraic k-theory and quadratic forms.
\newblock {\em Inventiones mathematicae}, 9(4):318--344, 1970.

\bibitem{minac1996witt}
J{\'a}n Min{\'a}c and Michel Spira.
\newblock Witt rings and galois groups.
\newblock {\em Annals of mathematics}, pages 35--60, 1996.

\bibitem{ribes2000profinite}
Luis Ribes and Pavel Zalesskii.
\newblock {\em Profinite groups}.
\newblock Springer, 2000.

\bibitem{wadsworth55merkurjev}
A~Wadsworth.
\newblock Merkurjev’s elementary proof of {M}erkurjev’s theorem.
\newblock In {\em Applications of Algebraic K-theory to Algebraic Geometry and
  Number Theory, Parts I, II}, volume~55, pages 741--776, Boulder, Colorado,
  1983. American Mathematical Society, Contemporary Mathematics.

\end{thebibliography}
\end{document}